%% file: wlarticle.tex
\documentclass[11pt]{article}
\usepackage[utf8]{inputenc}
\usepackage[english]{babel}
\usepackage{amsmath}
\usepackage{amsthm}
\usepackage{amssymb}
\usepackage{titling}
\usepackage{cite}
\usepackage{hyper}
\usepackage{url}
\usepackage[a4paper]{geometry}
\usepackage[T1]{fontenc}

\title{Limits of structures and Total NP Search Problems\protect\footnote{This work has been supported by Charles University Research Center program No.UNCE/SCI/022, the project SVV-2023-260721 and by the GA UK project No. 246223.}}
\author{Ondřej Ježil \\
\texttt{ondrej.jezil@email.cz}}
\date{Faculty of Mathematics and Physics, Charles University\protect\footnote{Sokolovská 83, Prague, 186 75, The Czech Republic}}

\begin{document}
\include{macros}

\maketitle
\begin{abstract}
    For an infinite class of finite graphs of unbounded size, we define a limit object, to be called a \emph{wide limit}, relative to some computationally restricted class of functions. The limit object is a first order Boolean-valued structure. The first order properties of the wide limit then reflect how a computationally restricted viewer ``sees'' a generic
    member of the class. The construction uses arithmetic forcing with
    random variables~\cite{krajicek2010forcing}. We give sufficient conditions for universal and existential sentences to be valid in the limit, provide several examples, and prove that such a limit object can then be expanded to a model of weak arithmetic. 
    
    To illustrate the concept we give an example in which the wide limit relates to total NP search problems. In particular, we take the wide limit of all maps from $\{0,\dots,k-1\}$ to $\{0,\dots,\floor{k/2}-1\}$ to obtain a
    model of $\forall \PV_1(f)$ where the problem $\OntoWeakPigeon$ is total but $\WeakPigeon$, the complete problem for $\PWPP$, is not. Thus, we obtain a new proof of this unprovability and show it implies that $\WeakPigeon$ is not many-one reducible to $\OntoWeakPigeon$ in the oracle setting.
\end{abstract}

\input{secintr.tex}

\input{secpreli.tex}

\input{secdef.tex}

\input{secmod.tex}

\input{sectfnp.tex}

\input{secsepwphp.tex}

\input{sectheory.tex}

\input{secconc.tex}

\section*{Acknowledgement}

This work is based on the author's master's thesis~\cite{jezil2022thesis} which was completed under the supervision of Jan Krajíček. The author thanks Eitetsu Ken for comments on a~draft of this paper.

\bibliography{wlbib}{}
\bibliographystyle{plain}
\end{document}

%% file: macros.tex
\def\rrbracket{{]\!]}}
\def\llbracket{{[\![}}
\newcommand{\PPA}{\textbf{PPA}}
\newcommand{\PPAD}{\textbf{PPAD}}
\newcommand{\LEAF}{\textsc{Leaf}}
\newcommand{\OntoWeakPigeon}{\textsc{RetractionWeakPigeon}}
\newcommand{\Pigeon}{\textsc{Pigeon}}
\newcommand{\SourceOrSink}{\textsc{SourceOrSink}}
\newcommand{\WeakPigeon}{\textsc{WeakPigeon}}
\newcommand{\EDGE}{\text{EDGE}}
\newcommand{\SK}{\text{SK}}
\newcommand{\Th}{\text{Th}}
\newcommand{\PV}{\text{PV}}
\newcommand{\s}{\texttt{s}}
\newcommand{\w}{\texttt{w}}
\renewcommand{\u}{\texttt{u}}
\newcommand{\p}{\texttt{p}}
\newcommand{\up}{\texttt{up}}
\newcommand{\LPV}{L_\text{PV}}
\newcommand{\WPHP}{\textbf{WPHP}}
\newcommand{\PWPP}{\textbf{PWPP}}
\newcommand{\PPP}{\textbf{PPP}}
\newcommand{\pPATH}{*\text{PATH}}
\newcommand{\Fnb}{F_\text{nb}}
\newcommand{\FPV}{F_\text{PV}}
\newcommand{\phie}{\varphi_{E}}
\newcommand{\phidegn}{\varphi_{\text{deg=0}}}
\newcommand{\phidegi}{\varphi_{\text{deg=1}}}
\newcommand{\phidegii}{\varphi_{\text{deg=2}}}
\newcommand{\phiLEAF}{\varphi_{\text{\LEAF}}}
\newcommand{\phiOWPHP}{\varphi_{\text{\OntoWeakPigeon}}}
\newcommand{\Tnb}{\T_\text{nb}}
\newcommand{\Tq}{\T_\text{q}}
\newcommand{\Mnh}{M_{n,\floor{n/2}}}
\newcommand{\Gnb}{G_\text{nb}}
\newcommand{\Fconst}{F_\text{const}}
\newcommand{\Frud}{F_\text{rud}}
\newcommand{\Trud}{\T_\text{rud}}
\newcommand{\0}{\textbf{0}}
\newcommand{\1}{\textbf{1}}
\newcommand{\Land}{\bigwedge}
\newcommand{\Lor}{\bigvee}
\newcommand{\A}{\mathcal{A}}
\newcommand{\B}{\mathcal{B}}
\newcommand{\I}{\mathcal{I}}
\newcommand{\C}{\mathcal{C}}
\newcommand{\G}{\mathcal{G}}
\newcommand{\M}{\mathcal{M}}
\newcommand{\T}{\mathcal{T}}
\renewcommand{\O}{\mathcal{O}}
\newcommand{\RR}{\mathbb{R}}
\newcommand{\NN}{\mathbb{N}}
\newcommand{\QQ}{\mathbb{Q}}
\newcommand{\ZZ}{\mathbb{Z}}
\newcommand{\bbl}{\llbracket}
\newcommand{\bbr}{\rrbracket}
\newcommand{\st}{\text{st}}
\newcommand{\new}{\text{new}}
\newcommand{\dom}{\text{dom}}
\newcommand{\TFNP}{\textbf{TFNP}}
\newcommand{\FP}{\textbf{FP}}
\newcommand{\NP}{\textbf{NP}}
\renewcommand{\P}{\textbf{P}}
\newcommand{\abs}[1]{\lvert#1 \rvert}
\newcommand{\floor}[1]{\lfloor #1 \rfloor}
\newcommand{\ceil}[1]{\lceil #1 \rceil}
\newtheoremstyle{mystyle}
  {}
  {}
  {}
  {}
  {\bfseries}
  {.}
  { }
  {}

\theoremstyle{mystyle}
\newtheorem{thrm}{Theorem}[section]
\newtheorem*{thrmm}{Theorem}
\newtheorem{lemm}[thrm]{Lemma}
\newtheorem*{lemmm}{Lemma}
\newtheorem{claim}[thrm]{Claim}
\newtheorem{defi}[thrm]{Definition}
\newtheorem{exam}[thrm]{Example}
\newtheorem{crll}[thrm]{Corollary}

\newtheorem{defn}{Definition}

\newtheorem*{cor}{Corollary}
\newtheorem*{rem}{Remark}
\newtheorem*{example}{Example}
\newtheorem*{prope}{Property}
\newtheorem*{prob}{Problem}

%% file: secintr.tex
\section{Introduction}
\addcontentsline{toc}{chapter}{Introduction}

The notion of limits of finite structures is prevalent both in logic and in combinatorics. In logic the examples are the ultraproduct and the compactness theorem, which was used in~\cite{Fagin1976} to prove the $0$--$1$ law for structures over relational vocabularies. In combinatorics the dense graph limit defined in \cite{lovasz2006limits} provided a framework to restate and find new proofs for results in extremal graph theory --- for instance Goodman's theorem relating the number of edges to the number of triangles in a graph. More notions of graph limits are discussed in~\cite{Nesetril2013}. Another recent use of limit objects for results of extremal combinatorics was by Razborov in~\cite{razborov2007flag}. 

In this work we define a new construction of a limit object. 
Given a class of finite graphs $\G$, whose vertex sets are initial segments of $\NN$, we can stratify it into the sequence of sets $\{\G_k\}_{k=1}^\infty$ as follows
\[\G_k=\{\omega\in \G; \omega\text{ has $\{0,\dots, k-1\}$ as its vertex set}\}.\]
We are interested in the case when the cardinalities of $\G_k$ are unbounded and hence our intended limit is a limit of sets of finite graphs. For this reason we call such a sequence of sets of graphs a \emph{wide sequence} and the limit object its \emph{wide limit}. The qualification wide refers to the fact that we are interested in sequences of sets of graphs rather than sequences of individual graphs.

We will give a full definition in Section \ref{secwidelimits}, but we describe some important properties here. For a class $F$ of functions (typically chosen to be a class with some computational restrictions) we define the wide limit denoted $\lim_{F}\G_n$, where $n$ is a technical parameter to be defined later. The wide limit $\lim_F \G_n$ is a Boolean-valued graph\footnote{Generally, we can do this with any $L$-structures for some first order language $L$. The limit object is then a Boolean-valued $L$-structure $\lim_{F}\G_n$. In this work we restrict ourselves to the language of graphs $L=\{E\}$ to simplify the presentation.} --- its edge relation does not only permit the truth values $\0$ and $\1$ but also many other values from some infinite complete Boolean algebra $\B$. This algebra is in fact also a measure algebra with a measure $\mu$ on it, so to any statement formulated as a first order sentence $\varphi$ we can assign a real number $\mu(\bbl\varphi\bbr)\in [0,1]$ which measures how far the truth value of $\varphi$ (denoted $\bbl\varphi\bbr$) is from the value $\0$. The key method we use is arithmetical forcing with random variables, developed in~\cite{krajicek2010forcing}, which allows us to construct models of (weak) arithmetical theories and by restricting to the language of graphs gives us Boolean-valued graphs. In these Boolean-valued graphs, existential statements which obtain the maximal truth-value $\1$ (valid sentences) correspond to the ability of $F$ to solve search problems over the class of graphs we are considering. 

Our limit object can be expanded to the original model that Krajíček's method would otherwise construct. We prove (Theorem~\ref{thrmtrans}) that the truth values of first order sentences concerning the object are preserved even when evaluated in the model of arithmetic relativized to the wide limit (under a mild condition on the family $F$). 

As an application of this construction, we reprove a many-one separation of black-box total NP search problems. A total NP search problem, first introduced in~\cite{PAPADIMITRIOU1994498}, is a problem which always has a solution of size polynomial in the input size and the correctness of its solution can be verified by a polynomial time machine. The qualifier black-box means that the input is given by a pair $(\O,x)$, where $x$ is a binary string and $\O$ is a function oracle. The problems of our interest are the following: The problem $\OntoWeakPigeon$, whose totality follows from a non-existence of a retraction from $\{0,\dots,k-1\}$ and $\{0,\dots,\floor{k/2}-1\}$. And the problem $\WeakPigeon$, whose totality follows from a non-existence of an injection between the same pair of sets.

We take a wide limit of all maps from the interval $\{0,\dots,k-1\}$ to the interval $\{0,\dots,\floor{k/2}-1\}$ relative to trees querying images of elements of subexponential depth (in the length of $k$) to obtain a Boolean-valued graph of a function $\lim_{\Fnb}\Mnh$, which is in a sense elementarily equivalent to an injective map from the interval $\{0,\dots,n-1\}$ to $\{0,\dots,\floor{n/2}-1\}$ for some non-standard number $n$. We then expand the wide limit into a model of arithmetic $K(\Mnh,\Fnb,\Gnb)$, where $\OntoWeakPigeon$ is total, but $\lim_{\Fnb}\Mnh$ is an instance of $\WeakPigeon$ which has no solution.

There is already an established connection between complexity of search problems and logic (namely witnessing theorems in bounded arithmetic, see~\cite{hanika2004thesis},~\cite{muller2021}). The model we construct is a model of relatively weak theory $\forall\PV_1(f)$, and also of open induction and open comprehension with parameters. The existence of this model gives a new proof of the result of Thapen~\cite[Theorem 4.10]{THAPEN2005}, that these principles along with the principle that $\OntoWeakPigeon$ is total cannot prove that the problem $\WeakPigeon$ is total.

This paper has two main parts. The emphasis of the paper is on the first \emph{conceptual} part, where we introduce the new notion of a wide limit. This part consists of Sections~\ref{secpreli}, \ref{secwidelimits} and \ref{secmodel}: In Section~\ref{secpreli} we recall preliminary notions, most importantly nonstandard models of arithmetic. In Section~\ref{secwidelimits} we give the definition of the wide limit, provide several examples and show how density of finite substructures corresponds to validity of existential sentences in the wide limit (Theorem~\ref{thrmsuffexist}). And in Section~\ref{secmodel} we prove that, under some reasonable assumptions, the wide limit can be expanded to a model of two sorted arithmetic (Theorem~\ref{thrmtrans}).

The second part, consisting of Sections~\ref{sectfnp} and \ref{secsepwphp}, is about our application of the new concept. In Section~\ref{sectfnp}, we recall the definition of query $\TFNP$ and show that the wide limit of all black-box instances of $\WeakPigeon$ is a graph of an injective map from $\{0,\dots,n-1\}$ to $\{0,\dots,\floor{n/2}-1\}$. Finally, in Section~\ref{secsepwphp}, we expand the wide limit to a model of two-sorted arithmetic, $K(\Mnh,\Fnb,\Gnb)$, where we verify the following properties: 
\begin{itemize}
    \item (Theorem~\ref{thrmwphp}): the wide limit is an instance of $\WeakPigeon$ without a solution,
    \item (Theorem~\ref{thrmrwphp}): every instance of $\OntoWeakPigeon$ has a solution,
    \item (Theorem~\ref{thrmopenind}): open comprehension and open induction,
    \item (Theorem~\ref{thrmpv}): the theory $\forall \PV_1(f)$,
\end{itemize}
and also refute the existence of a many-one reduction from the problem $\WeakPigeon$ to the problem $\OntoWeakPigeon$ (Theorem~\ref{thrmsep}). This separation is not new and follows from known results, see Section~\ref{secconc} for more details, but our illustration gives the separation a clear semantic interpretation.


%% file: secpreli.tex
\section{Preliminaries} \label{secpreli}

By graphs, we mean structures in a language with a single binary relation denoted $E$ which is antireflexive and symmetric as we only consider undirected graphs in this work. We will denote any particular graph by $\omega$ as it will be used in some sense as a sample of a discrete probability space. The edge relation of a particular graph $\omega$ will be denoted $E_\omega$.

In the rest of this section, we recall notions needed for Krajíček's forcing construction. A fundamental notion we use throughout the work is of nonstandard models of (true) arithmetic. Let $L_{all}$ be the language containing the names of all relations and functions on the natural numbers and let $\Th_{L_{all}}(\NN)$ denote the set of true sentences in this language in the standard model $\NN$. By classical results of logic there exist $L_{all}$-structures in which all sentences from $\Th_{L_{all}}(\NN)$ are valid but which are not isomorphic to $\NN$. These are called \emph{nonstandard models} (of $\Th_{L_{all}}(\NN)$).

All nonstandard models of $\Th_{L_{all}}(\NN)$ (and even much weaker theories) contain an isomorphic copy of $\NN$ as an initial segment. Therefore, we can assume that in fact all models we encounter satisfy $\NN\subseteq \M$. After considering a concrete nonstandard model $\M$ (of $\Th_{L_{all}}(\NN)$) we shall call the elements of $\M\setminus\NN$ \emph{nonstandard numbers}. These can be intuitively understood as ``infinite natural numbers''. The key feature of those elements is that all functions and relations from $L_{all}$ are defined even on nonstandard numbers. This includes functions for coding sequences and sets by numbers, and therefore we can use notations like $a_0,\dots,a_{n-1}$ even for a nonstandard number $n$. The notation then means that for each $i\in\M$ such that $i<n$ we have an object $a_i$ coded by a number in $\M$ and that this whole sequence is coded by some number ${\{a_i\}}_{i=0}^{n-1}\in\M$. For a nonstandard number $S\in\M$ coding a set we denote its nonstandard size (cardinality) by $\abs{S}$. In the case where we talk about a binary string $x$ the notation $\abs{x}$ denotes the length of $x$ (which is nonstandard if $x$ is), and if $m$ is simply an element of $\M$ we denote $\abs{m}$ its bit length. The symbol $\infty$ in sequences $\{a_i\}_{i=0}^\infty$ implies indexing over standard natural numbers, and in limits $\lim_{k\to\infty}(\dots)$ has the standard meaning.

In the next section we will fix a nonstandard model $\M$ which has the model theoretic property of being $\aleph_1$-saturated. There is a self-contained construction of such model in~\cite[Appendix]{krajicek2010forcing}, but we never explicitly use the $\aleph_1$-saturation. The reason we require it is that without it, it might happen that the Boolean algebra we construct is not complete. A property, implied by our choice of language, which we do use explicitly is the following:

\begin{prope}
    For any sequence ${\{a_i\}}_{i=0}^\infty$ of standard natural numbers, there is a function symbol $f$ in $L_{all}$ such that for all $i\in\NN$ we have $\M \models f(i)=a_i$. This gives rise to a sequence $\{f(i)\}_{i\in \M}$ which agrees with $\{a_i\}_{i=0}^\infty$ on elements with standard index, but also has elements with index for any $i\in\M$. We call $\{f(i)\}_{i\in\M}$ the \emph{nonstandard prolongation} of $\{a_i\}_{i=0}^\infty$ and for any $m\in\M$ we will denote $f(m)$ as simply $a_m$. 
\end{prope}

Using this property, we will now allow ourselves to freely use any elements of $\M$ as indices of sequences of standard numbers and generally any sequences of standard finite objects, which can indeed be coded by standard natural numbers.

Any nonstandard model $\M$ can be extended to an ordered ring $\ZZ^\M$ by adding negative elements. This ring then can be extended to a fraction field $\QQ^\M$. We shall call elements of $\QQ^\M$ $\M$\emph{-rationals}. The field $\QQ^\M$ contains an isomorphic copy of $\QQ$ as a substructure. 
We shall use the structure $\QQ^\M$ analogously to how hyperreal numbers are used in nonstandard analysis. For more details about nonstandard analysis we recommend~\cite{goldbring2014lecture} to the interested reader. 
We call an element in $\QQ^\M$ with absolute value greater than all $\frac{k}{1},k\in\NN$, \emph{infinite}, otherwise we call it \emph{finite}. We call elements in $\QQ^\M$ with absolute value smaller than all $\frac{1}{k},k\in\NN$, \emph{infinitesimal}.
We will denote the set of finite $\M$-rationals as $\QQ^\M_{fin}$ and one can check it forms an ordered ring.

\begin{lemm}[\protect{{The existence of a standard part~\cite[Theorem 1.9]{goldbring2014lecture}}}]
    There is a function $\st:\QQ^\M_{fin}\to \RR$ assigning to each finite $\M$-rational a real number. The function $\st$ is a ring homomorphism and the kernel of $\st$ is exactly the ideal of infinitesimal numbers. When $q$ is a finite $\M$-rational we call $\st(q)$ its \emph{standard part}.
\end{lemm}

The following result characterizes convergence of sequences of rational numbers using the $\M$-rational numbers $\QQ^\M$.

\begin{thrm}[\hspace{-0.02em}{{\cite[Theorem 3.2]{goldbring2014lecture}}}]\label{thrmnonstdanal}
    Let $\{a_i\}_{i=0}^\infty$ be a sequence of standard rational numbers and let $r\in\RR$. Then the following are equivalent.
    \begin{itemize}
        \item $\lim_{i\to\infty} a_i = r$
        \item For every nonstandard $t \in \M\setminus \NN$ we have: $\st(a_t)=r$.
    \end{itemize}
\end{thrm}

This theorem shows that computing a standard part of an $\M$-rational obtained from a nonstandard prolongation of some standard sequence is equivalent to computing the limit of this sequence. It will be apparent that computations emerging in the analysis of wide limits can be completed in a purely standard setting --- by computing limits of specific probabilities associated with the wide sequence we consider. In this work, we mostly present the computations with nonstandard parameters, this seems to be natural in the context of our limit object and in some sense provides a corresponding shift in perspective: Instead of directly analyzing the sequence we start with, we perform calculations with the nonstandard numbers which enumerate elements at a concrete nonstandard index of the sequence. A reader interested in interpreting those computations in the standard setting can in most cases simply imagine the manipulations to be prefixed with: 
\begin{center}
``For all sufficiently large standard $n$ \dots''
\end{center}

It is important for arithmetical forcing with random variables to consider discrete probability spaces of nonstandard size. We shall always use the uniform distribution on the samples, although this is not necessary for the general construction. Thus, the probability of an event coded by an element $A\in \M$ is then just the $\M$-rational number $\abs{A}/\abs{S}$ where $S$ is the set of all samples.

We conclude this section by restating classical inequalities used in this work using the nonstandard approach. Their validity follows from the usual inequalities and Theorem~\ref{thrmnonstdanal}.

\begin{thrm}[Bernoulli's inequality]\label{thrmbrnl}
    Let $y \in \M$, $x\in \QQ^\M$ and $x\geq -1$, then
    \[(1+x)^y \geq 1+yx.\]
    \end{thrm}
    
    \begin{thrm}[Exponential equality]\label{thrmexpo}
    Let $x \in \M\setminus \NN$, then
    \[\st\left(\left(1-\frac{1}{x}\right)^x\right) = e^{-1}.\]
\end{thrm}

%% file: secdef.tex
\section{Wide Limits}
\label{secwidelimits}

We shall define a wide limit of every sequence of the following form.

\begin{defi}
    A sequence of sets of graphs $\{\G_k\}_{k=1}^\infty$ is called \emph{a wide sequence} if the following holds:
    \begin{itemize}
        \item Every graph $\omega\in\G_k$ has the vertex set $\{0,\dots,k-1\}$.
        \item $\lim_{k\to\infty}\abs{\G_k}=\infty$.
    \end{itemize}
\end{defi}

By abuse of notation we will simply talk about a wide sequence $\G_k$ instead of $\{\G_k\}_{k=1}^\infty$. Since a wide limit is a Boolean-valued graph, we need to construct a Boolean algebra in which the truth evaluation of statements shall take place. For the construction of the Boolean algebra we will closely follow~\cite[Chapter 1]{krajicek2010forcing} albeit with slight changes. Let us fix now, for the rest of this work, an $\aleph_1$-saturated model of $\Th_{L_{all}}(\NN)$ which we will denote $\M$ and with it we fix a nonstandard number $n\in \M$.

\begin{defi}
    We define \[\A=\{A\subseteq \{0,\dots,n-1\}; A\in\M\},\]
    in words $\A$ is the set of subsets of $\{0,\dots,n-1\}$ coded by an element in $\M$. This is a Boolean algebra and to each $A\in\A$ we assign an $\M$-rational $\abs{A}/n$ which we call its \emph{counting measure}.
\end{defi}

Even though $\A$ is a Boolean algebra with a ``measure'' it is not a $\sigma$-algebra. Indeed, $\A$ contains all singletons $\{k\}$ for $k$ standard, but the countable union of those singletons, the set of standard natural numbers $\NN$, is not definable by overspill. However, if we had joins and meets of arbitrary subsets at our disposal it would allow us to interpret quantifiers in the Boolean-valued case, so we now want to `tweak' this Boolean algebra.

\begin{defi}
    Let $\I$ be the ideal of $\A$ consisting of elements with infinitesimal counting measure. We define $\B=\A/\I$. Each element in $\B$ is of the form $A/\I$, where $A\in\A$, and we define $\mu(A/\I)=\st(\abs{A}/n)$. We will denote the maximal element of $\B$ by $\1$ and the minimal element by $\0$.
\end{defi}

One can easily check that $\mu$ is well-defined since for all $A\in\I$ it holds that $\st(\abs{A}/n)=0$. The measure $\mu$ is called the Loeb measure. Relying on the $\aleph_1$-saturation of $\M$, the following then holds.

\begin{lemm}[\hspace{-0.02em}{{\cite[Lemma 1.2.1]{krajicek2010forcing}}}]
    $\B$ is a $\sigma$-algebra with a real valued measure $\mu$. Moreover, $\B$ is a complete Boolean algebra.
\end{lemm}

It is important to note that $\1\in \B$ is the only element of $\B$ with measure $\mu(\1)=1$ and similarly $\0\in\B$ is the only element with measure $\mu(\0)=0$. Also, for $B,B'\in \B$ the inequality $B\leq B'$ implies $\mu(B) \leq \mu(B')$.

We now define precisely what we mean by the family of functions $F$ relative to which we will be taking the wide limit. This is still a part of Krajíček's construction, we just modify it to make it compatible with our setup that starts with a wide sequence.

For every $k\in\NN$ the set $\G_k$ is finite and thus can be coded by a standard number. Therefore, there is a nonstandard prolongation of this sequence, and we can consider the set coded by the nonstandard number $\G_n$, which matches the value of the function symbol in $L_{all}$ describing the function $k \mapsto \G_k$ when $k=n$. 

\begin{defi}
    Let $\{\G_k\}_{k=1}^\infty$ be a wide sequence. We say that $F$ is \emph{a family of random variables on} $\G_n$ if every $\alpha\in F$ is a function coded by a number in $\M$ with domain $\G_n$ and taking values in $\M$. We say $\alpha\in F$ is an $F$-vertex if for all $\omega\in\G_n$ it holds that $\alpha(\omega)\in\{0,\dots,n-1\}$. The set of all $F$-vertices is denoted $U(F)$.
\end{defi}

If the wide sequence $\{\G_k\}_{k=1}^\infty$ is clear from context we just say $F$ is a family of random variables. This is for now everything we need to recall from~\cite{krajicek2010forcing}, and we can proceed to define the central object of our work.

\begin{defi}[The wide limit]
    Let $\{\G_k\}_{k=1}^\infty$ be a wide sequence and let $F$ be a family of random variables on $\G_n$. We define \emph{the wide limit} $\lim_F\G_n$ as a $\B$-valued structure in the language consisting of a single binary relation symbol $\{E\}$ as follows. The universe of the wide limit is taken as the set of all $F$-vertices. We now inductively define the truth values for all $\{E\}$-sentences.
    \begin{itemize}
        \item $\bbl \alpha=\beta \bbr=\{\omega\in\G_n; \alpha(\omega)=\beta(\omega)\}/\I$
        \item $\bbl E(\alpha,\beta) \bbr=\{\omega\in\G_n; E_{\omega}(\alpha(\omega),\beta(\omega))\}/\I$
        \item $\bbl-\bbr$ commutes with $\lnot$, $\land$ and $\lor$ 
        \item $\bbl (\exists x)A(x)\bbr=\Lor_{\alpha\in U(F)}\bbl A(\alpha)\bbr$
        \item $\bbl (\forall x)A(x)\bbr=\Land_{\alpha\in U(F)}\bbl A(\alpha)\bbr$
    \end{itemize}
\end{defi}

To stress in which Boolean-valued structure is the truth evaluation $\bbl-\bbr$ taking place we will sometimes denote the evaluation $\C_1\bbl-\bbr$, $\C_2\bbl-\bbr$ for Boolean-valued structures $\C_1$ and $\C_2$ respectively. Furthermore, if $\C_1\bbl\varphi\bbr = \1$ for some sentence $\varphi$ we say $\varphi$ is \emph{valid} in $\C_1$.

Note that since $\G_n$ can be recovered from $F$ as the domain of its elements, the wide limit only depends on $F$, strictly speaking. We keep $\G_n$ in the notation to cover the situation where we have a very general family of functions (e.g. the family of polynomial functions $\FPV$) which can be applied to every wide sequence. Thus, the notation $\lim_{F}\G_n$ means that $F$ is restricted to those functions which take elements of $\G_n$ as an input even when $F$ possibly contains other functions too.

The potential variability of the parameter $n$ may also seem unnecessary and indeed in this section it is, but in Section~\ref{secsepwphp} we will assume that $n$ is a power of two, which will allow us to more easily translate between the results about wide limits and reducibility of search problems.

\subsection{Wide Limits for Shallow Decision Trees}

Now we shall define the first nontrivial family of random variables relative to which we shall take wide limits of several sequences. The functions in the family will be computed by shallow decision trees. So the shape of the wide limit reflects what can a tree of depth subexponential in $\abs{n}$ witness in the wide sequence with probability arbitrarily close to~$1$.

\begin{defi}
    Let $\Trud$ be a family\footnote{The subscript `rud' stands for rudimentary. The name for the family is taken from \cite{krajicek2010forcing}.} of labeled rooted binary trees in $\M$ of the following form.
    At each vertex the tree is labeled by an element of $\{0,\dots,n-1\}\times\{0,\dots,n-1\}$ and the two outgoing edges incident to it are labeled as $0$ and $1$ respectively. The leaves are labeled by an element of $\M$. The depth of the tree is bounded by a number of a form $n^{1/t}$ (rounded to the nearest element of $\M$) for some $t\in\M\setminus \NN$.

    A \emph{computation} of a $T\in\T_{rud}$ on some $\omega\in\G_n$ is defined as follows. Start at the root and interpret each label $(i,j)$ of the vertex as a question whether the pair $(i,j)$ is in the edge set $E_\omega$ and follow a path through $T$ reading $1$ as a positive answer and $0$ as a negative answer. The label of the leaf visited at the end of the path is the output of $T$ on $\omega$, denoted $T(\omega)$.
    
    We define $\Frud$ to be the set of all functions computed by a tree $T\in\Trud$.
\end{defi}

In the following example, we consider a simple wide sequence of sets of graphs with exactly one edge.

\begin{exam}\label{examedgeempty} Let $\EDGE_k=\{(\{0,\dots,k-1\},E);\abs{E}=1\}$. Since any $\omega\in\EDGE_k$ has only $1$ edge in all potential $k\cdot(k-1)/2$ edges, it is not likely a shallow tree will find the edge. This is the idea behind the proof of the following claim.

    \[\lim_{\Frud}\EDGE_n\bbl(\exists x)(\exists y)E(x,y)\bbr=\0\]
    Let $\alpha,\beta\in U(\Frud)$, we proceed by proving that \[\bbl E(\alpha,\beta)\bbr=\0\] which is enough to prove the theorem since \[\bbl(\exists x)(\exists y)E(x,y)\bbr=\Lor_{\alpha\in U(\Frud)}\Lor_{\beta\in U(\Frud)}\bbl E(\alpha,\beta)\bbr=\Lor_{\alpha\in U(\Frud)}\Lor_{\beta\in U(\Frud)}\0=\0.\]

    Let $\alpha$ and $\beta$ be computed by $T\in\Trud$ and $S\in\Trud$ respectively. Let the depth of both $T$ and $S$ be at most $n^{1/t}$, where $t\in\M\setminus \NN$. Walk down $T$ from the root and always prolong the path along the edge labeled $0$. On this path we have a set of at most $n^{1/t}$ different pairs of vertices as the labels of edges and a label of the leaf node $l_T$.

    We do the same for $S$, and we find another set of at most $n^{1/t}$ pairs of vertices and a label of the leaf $l_S$. The labels $l_S$ and $l_T$ are then combined to one last pair $\{l_S,l_T\}$. Now we just need to compute the probability that none of these $2n^{1/t}+1$ pairs of vertices are in the edge set $E_\omega$. 

    There are $\binom{n}{2}$ different graphs in $\EDGE_n$ and at least $\binom{n-4n^{1/t}-2}{2}$ graphs which fulfill our requirements, namely, those graphs whose sole edge is not incident with the vertices from the labels of the trees $S$ and $T$. The probability is by Theorem~\ref{thrmbrnl} at least
    \begin{align*}
        \frac{\binom{n-4n^{1/t}-2}{2}}{\binom{n}{2}}&=\frac{(n-4n^{1/t}-2)(n-4n^{1/t}-3)}{n(n-1)}\\
        &\geq \left(1-\frac{8n^{1/t}+6}{n}\right)
    \end{align*}
    after taking the standard part of the last line we get $\st(1-\frac{8n^{1/t}+6}{n})=1$. Therefore, $\mu(\bbl E(\alpha,\beta)\bbr)=0$ and $\bbl E(\alpha,\beta)\bbr=\0$.

\end{exam}

\subsection{Sufficient conditions for validity of universal and existential sentences}

To understand wide limits we need to compute the truth-values of sentences which describe properties whose complexity we are interested in. Generally, for sentences of arbitrary complexity, this can be hard. In this section we prove sufficient conditions at least for the validity of universal and existential sentences. 

We will start with the simpler condition for the validity of universal sentences. This is important also because we would like to know that a wide limit of a wide sequence of graphs is also a graph: Meaning the statement that $E$ is antireflexive and symmetric is valid in the wide limit, and this statement is indeed expressible as a universal sentence.

\begin{thrm}\label{thrmsuffuniv}
    Let $\G_k$ be a wide sequence and let $F$ be any family of random variables. Let $\varphi(x_0,\dots, x_{l-1})$ be an open $\{E\}$-formula and assume that 
    \[\lim_{k\to\infty}\Pr_{\omega \in \G_k}[\omega \models (\forall \overline x)\varphi(\overline x)]=1.\]

    Then $\lim_F\G_n\bbl(\forall \overline x)\varphi(\overline x)\bbr=\1$.
    
\end{thrm}
\begin{proof}


    By the assumption and Theorem \ref{thrmnonstdanal} we get that $\st(\Pr_{\omega\in\G_n}[\omega\models (\forall \overline x)\varphi(x)])=1$.

    Since $\varphi$ is open, we have for every tuple of $F$-vertices $\overline \alpha$ that $\bbl \varphi(\overline \alpha)\bbr=\1$. Now 
    \begin{align*}
        \bbl(\forall x)\varphi(x)\bbr&=\Land_{\overline \alpha\in U(F)^l}\bbl\varphi(\overline \alpha)\bbr\\
        &=\Land_{\overline \alpha\in U(F)^l}\1\\
        &=\1.
    \end{align*}
    
    \vspace{-2em}
\end{proof}

\begin{crll}
    Let $\G_k$ be a wide sequence and $F$ any family of random variable, then $\lim_F \G_n$ is an $\{E\}$-structure in which both antireflexivity and symmetry of $E$ is valid (i.e. $\lim_{F}$ is a Boolean-valued graph).
\end{crll}

Now to give a sufficient condition for the validity of an existential sentence $(\exists \overline x)\varphi(\overline x)$ we use the auxiliary value of \emph{density of $\varphi(x_0,\dots,x_{l-1})$} defined as the probability that a random graph $\omega\in\G_k$ and a random tuple $\overline a \in \{0,\dots,k-1\}^{l}$ satisfy $\omega\models\varphi(\overline a)$ and show that the limiting density gives a lower bound for the measure of $\bbl(\exists \overline x)\varphi(\overline x)\bbr$.

\begin{thrm}\label{thrmsuffexist}
    Let $\G_k$ be a wide sequence and let $F$ be a family of random variables which contains all constant functions. Let $\varphi(x_0,\dots,x_{l-1})$ be an open $\{E\}$-formula and let $p\in [0,1]$. Assume that
    \[\lim_{k\to\infty} \Pr_{\substack{\omega\in \G_k\\\overline a}}[\omega\models \varphi(\overline a)]\geq p,\]
    where $\overline a$ is sampled uniformly over all elements of $\{0,\dots,k-1\}^l$. Then \[\mu(\lim_F\G_n\bbl(\exists \overline x)\varphi(\overline x)\bbr)\geq p.\] 
    In particular if $p=1$ then $\lim_F\G_n\bbl(\exists \overline x)\varphi(\overline x)\bbr=\1.$
\end{thrm}
\begin{proof}
    Consider an array $C$ indexed by $\omega\in \G_n$ and $\overline a \in \{0,\dots,n-1\}^l$ such that
     \[C_{\omega,\overline a}=\begin{cases}
         1&\omega \models \varphi(\overline a)\\
         0&\text{otherwise.}
     \end{cases}\]

     By the assumption and induction in $\M$ we have that
     \[\st\left(\frac{1}{n^l\abs{\G_n}}\sum_{\omega\in\G_n}\sum_{\overline a}C_{\omega,\overline a}\right)\geq p.\]
     
     We now claim that there exists a specific $\overline b\in\{0,\dots,n-1\}^l$ such that $\st(\Pr_{\omega\in\G_n}[\omega \models \varphi(\overline b)])\geq p$. Assume for contradiction that the claim is false. Then
     \begin{align*}
        \frac{1}{\abs{\G_n}n^l}\sum_{\omega\in\G_n}\sum_{\overline a}C_{\omega,\overline \alpha}&=
        \frac{1}{n^l}\sum_{\overline a}\Pr_{\omega\in\G_n}[\omega\models \varphi(\overline a)]\\
        &\leq \Pr_{\omega\in\G_n}[\omega\models\varphi(\overline a_0)],
    \end{align*}
    where we pick\footnote{This is possible because $\M$ being a model of $\Th(\NN)$ satisfies induction.} $\overline a_0$ such that it maximizes $\Pr_{\omega\in\G_n}[\omega\models\varphi(\overline a_0)]$. But after taking the standard part of the inequality we obtain that
     \[\st\left(\frac{1}{n^l\abs{\G_n}}\sum_{\omega\in\G_n}\sum_{\overline a}C_{\omega,\overline a}\right)\leq \st(\Pr_{\omega\in\G_n}[\omega\models\varphi(\overline a_0)])<p,\]
    which is a contradiction and so the claim is true. Let $\overline\gamma_b$ be a tuple of constant functions which is at every sample equal to $\overline b$. We have 
    \begin{align*}
        \bbl(\exists \overline x)\varphi(\overline x)\bbr &= \Lor_{\overline \alpha \in U(F)^l}
        \bbl\varphi(\overline \alpha)\bbr\\
        &\geq \bbl\varphi(\overline \gamma_b)\bbr
    \end{align*}
    and by taking $\mu$ of this inequality we finally obtain that $\mu(\bbl(\exists \overline x)\varphi(\overline x)\bbr)\geq p$.
\end{proof}

The following example demonstrates that Theorem~\ref{thrmsuffuniv} cannot be generalized to a similar hypothesis as Theorem~\ref{thrmsuffexist}.

\begin{exam}
    Let $\G_k$ consist of all undirected graphs on the vertex set $\{0,\dots,k-1\}$ with exactly $\ceil{\frac{k(k-1)}{2\log(k)}}
    $edges. One can see that
    \[\lim_{k\to\infty}\Pr_{\substack{\omega\in \G_k\\x,y}}[\omega \models \lnot E(x,y)]=1,\]
    but in fact $\lim_{\Frud}\G_n\bbl(\forall x)(\forall y)\lnot E(x,y)\bbr=\0$.
    
    Let $t\in\M\setminus \NN$ such that $n^{1/t}$ is not bounded above by a standard number. Let $T$ be a tree which queries on all paths a fixed set of $n^{1/t}$ different potential edges. If we prove that any such set in $\G_n$ has to contain at least one edge with probability infinitesimally close to $1$ then we can construct $\Frud$-vertices $\alpha$ and $\beta$ using $T$ such that $\bbl E(\alpha,\beta)\bbr=\1$ by simply taking $T$ and labeling each leaf on a path which finds an edge with either the lesser vertex (in the natural order of $\M$) to compute $\alpha$, or with the greater vertex to compute $\beta$.

    Let $S$ be the set of potential edges queried by $T$ and let $m=\binom{n}{2}$. Now we have
    \begin{align*}
        \Pr_{\omega\in\G_n}[\text{$S$ contains no edge in $\omega$}]&=\frac{(m-n^{1/t})!(m-\ceil{\frac{m}{\log n}})!}{m!(m-\ceil{\frac{m}{\log n}}-n^{1/t})!}\\
        &=\prod_{i=0}^{n^{1/t}-1}\frac{m-\ceil{\frac{m}{\log n}}-i}{m-i}\\
        &\leq \left(1-\frac{\ceil{\frac{m}{\log n}}}{m}\right)^{n^{1/t}}\\
        &\leq \left(1-\frac{1}{2\log n}\right)^{n^{1/t}},
    \end{align*}
    the standard part of which can be, using Theorem~\ref{thrmexpo}, for all $k\in \NN$ bounded above by 
    \begin{align*}
        \st\left(\left(1-\frac{1}{2\log n}\right)^{k\cdot 2 \log n}\right)&=e^{-k},
    \end{align*}
    which tends to $0$ as $k\to\infty$.
\end{exam}
    
For more examples, we point the interested reader to~\cite{jezil2022thesis}.

%% file: secmod.tex
\section{Expanding Wide Limits to Two-Sorted Arithmetic}\label{secmodel}

In this section, we will show that under reasonable assumptions one can embed the wide limit into the original models of Krajíček in such a way that first order statements, after a suitable translation, keep their truth values.

\subsection{The structures $K(F,G)$}

We will now recall the construction of two-sorted models of weak arithmetic $K(F,G)$ defined in~\cite[Chapter 5]{krajicek2010forcing}. We will take the liberty to define them as an extension of the definition of a wide limit to obtain structures $K(\G_n,F,G)$~\footnote{This notation is just making some parameters of the construction explicit, the models constructed can be obtained by the original method without first constructing the wide limit. Our contribution is in observing that the truth values of first order sentences concerning the wide limit is preserved between the wide limit and the structure $K(\G_n,F,G)$.}. Under the right conditions, these result in a structure in some sublanguage of $L_{all}$ with two sorts: numbers and functions of bounded domain on numbers, and this latter sort contains the wide limit as an object. These sorts are to be named number sort and set sort, as the bounded functions can be interpreted as sets (or more generally relations) and every such function $f$ can be coded by a bounded set $\{(x,i); \text{ The $i$-th bit of $f(x)$ is $1$.}\}$.

\begin{defi}
    Let $L\subseteq L_{all}$. This determines a \emph{language $L^{2}$} which we get by keeping the original variables as number sort variables, adding to $L$ set sort variables $X,Y,\dots$ whose intended interpretation are bounded functions and the equality symbol for set sort variables (denoted the same as for the number sort). All set sort variables are treated as function symbols and can form terms with the number sort terms as arguments.
\end{defi}

We will also use the function sort variables as relation symbols, and we define the atomic formula $X(x_0,\dots,x_{k-1})$ to be evaluated the same as the formula \[X(x_0,\dots,x_{k-1})\neq 0.\] We will now fix a wide sequence $\G_k$ and a family of random variables $F$ on $\G_n$ which together determine a wide limit $\lim_{F}\G_n$.

\begin{defi}
    We define $\M_n\subseteq \M$ to be the subset of $\M$ consisting of all numbers bounded above by $2^{n^{1/t}}$ for some $t\in\M\setminus\NN$. 
\end{defi}
\begin{defi}
    We define $L_n\subseteq L_{all}$ to contain all relation symbols from $L_{all}$ and all functions from $L_{all}$ for which their values on any tuple of elements of $\M_n$ is still in $\M_n$. We say $F$ is $L_n$-closed if for every function symbol $f\in L_n$ of arity $k$ and every $\alpha_0,\dots,\alpha_{k-1}\in F$ we have that $f(\alpha_{0},\dots,\alpha_{k-1})\in F$.

Note that $\M_n$ is then a substructure of the $L_n$-reduct of $\M$.
\end{defi}

\begin{defi}
    We say that $G$ is a \emph{family of random functions} (on $\G_n$) if for every $\Theta\in G$ there is $k\in\NN$ such that $\Theta$ assigns to each $\omega\in\G_n$ a $k$-ary function $\Theta_\omega$ coded by an element in $\M$ which maps a set $\dom(\Theta_\omega)\subseteq \M_n$ into $\M_n$. Such $\Theta$ is then called $k$-ary.

    We say $G$ is $F$-compatible if for every $\alpha_0,\dots, \alpha_{k-1}\in F$, and a $k$-ary $\Theta \in G$ we have that the function
    $\Theta(\alpha_0,\dots,\alpha_{k-1})$ defined as
    \[\Theta(\alpha)(\omega)=\begin{cases}
        \Theta_\omega(\alpha_0(\omega),\dots,\alpha_{k-1}(\omega))&\text{if $(\alpha_0(\omega),\dots,\alpha_{k-1}(\omega))\in \dom(\Theta_\omega)$}\\
        0&\text{otherwise,}
    \end{cases}\]
    is in fact in $F$.
\end{defi}

An example of a specific family of random functions will be provided in Section~\ref{secsepwphp}.

\begin{defi}
    Let $F$ be an $L_n$-closed family of random variables with values in $\M_n$. Let $G$ be an $F$-compatible family of random functions. We define $K(\G_n,F,G)$ to be a $\B$-valued $L_n^2$-structure with the number sort of the universe as $F$ and the set sort of the universe as $G$. The valuation of formulas is then given by the following inductive definition. We define the valuation only for terms consisting of variables substituted with elements of $F$, and because of the $L_n$-closedness and $F$-compatibility of $G$, we can evaluate more complex terms step-by-step, eventually reaching an element of $F$.
    \begin{itemize}
        \item $\bbl\alpha = \beta\bbr=\{\omega\in\G_n;\alpha(\omega)=\beta(\omega)\}/\I$, where $\alpha,\beta\in F$
        \item $\bbl R(\alpha_0,\dots,\alpha_{k-1})\bbr=\{\omega\in\G_n;\M_n\models R(\alpha_0(\omega),\dots,\alpha_{k-1}(\omega))\}/\I$, where the elements $\alpha_0,\dots,\alpha_{k-1}$ are from $F$ and $R$ is a relation symbol in $L_n$
        \item $\bbl\Theta=\Xi\bbr=\{\omega\in\G_n;\Theta_\omega=\Xi_\omega\}/\I$, where $\Theta,\Xi\in G$
        \item $\bbl-\bbr$ commutes with $\lnot$, $\land$ and $\lor$ 
        \item $\bbl(\forall x)A(x)\bbr=\Land_{\alpha\in F}\bbl A(\alpha)\bbr$
        \item $\bbl(\exists x)A(x)\bbr=\Lor_{\alpha\in F}\bbl A(\alpha)\bbr$
        \item $\bbl(\forall X)A(X)\bbr=\Land_{\Theta\in G}\bbl A(\Theta)\bbr$
        \item $\bbl(\exists X)A(X)\bbr=\Lor_{\Theta\in G}\bbl A(\Theta)\bbr$.
    \end{itemize}
\end{defi}

Let us note that in general, it is possible that extensionality is not valid in $K(\G_n,F,G)$.

\subsection{Preservation of sentences concerning the wide limit}

We will now prove (under a mild condition on $F$) that there is a set sort object in $K(\G_n,F,G)$ which faithfully encodes the wide limit $\lim_F\G_n$. This lets us construct models in which an object with the same first order properties as the wide limit might be desired. Recall that every element $\Theta \in G$, where $G$ is a family of random functions, actually determines a predicate which we evaluate the same as the formula $\Theta(\alpha_0,\dots,\alpha_{k-1})\neq 0$.

\begin{defi}\label{defirepre}
    Let $G$ be a family of random functions. We say that the edge relation of the wide limit $\lim_{F}\G_n$ is \emph{represented in $G$ by $\Gamma$} if $\Gamma\in G$ is binary and for all $\alpha,\beta\in U(F)$ we have that
\[K(\G_n,F,G)\bbl\Gamma(\alpha,\beta)\bbr=\lim_{F}\G_n\bbl E(\alpha,\beta)\bbr.\]
\end{defi}

\begin{defi}\label{defirestri}
    We say a family of random variables $F$ has \emph{restrictable ranges} if for every $\alpha\in F$ and $m\in\M_n$ there is $\tilde \alpha_m \in F$ such that
    \[\tilde \alpha_m(\omega) = \begin{cases}
        \alpha(\omega) & \alpha(\omega)<m\\
        0&\text{otherwise.}
    \end{cases}\]
\end{defi}

\begin{thrm}\label{thrmtrans}
   Let $\varphi$ be a $\{E\}$-sentence. Let $F$ be $L_n$-closed and have restrictable ranges and let $G$ be $F$-compatible. Let the edge relation of the wide limit $\lim_{F}\G_n$ be represented in $G$ by $\Gamma$. We define $\tilde \varphi(\Gamma)$ to be the $L_n^2$-sentence obtained by replacing all the occurrences of the relation symbol $E$ by $\Gamma$, keeping the structure of the logical connectives and replacing all quantifiers $(\forall x)(\dots)$ by $(\forall x)(x<n\to(\dots))$ and $(\exists x)(\dots)$ by $(\exists x)(x<n \land \dots)$.

   Then we have that
   \[\lim_{F}\G_n\bbl\varphi\bbr=K(\G_n,F,G)\bbl\tilde \varphi(\Gamma)\bbr.\]
\end{thrm}
\begin{proof}
    We will prove that for all $\{E\}$-formulas $\varphi(\overline x)$ and all $\overline \alpha\in U(F)$ we have that
   \[\lim_{F}\G_n\bbl\varphi(\overline \alpha)\bbr=K(\G_n,F,G)\bbl\tilde \varphi(\Gamma,\overline \alpha)\bbr.\]
   
   We proceed by induction on the complexity of the formula. The case for atomic formulas follows from Definition~\ref{defirepre} and the step for logical connectives is also clear since $\bbl-\bbr$ commutes with them. With the induction step for negation in hand it is now enough to prove the induction step for the universal quantifier. 
   
   We now assume that the statement holds for a formula $\varphi(y,\overline x)$. By the restrictability of ranges in $F$ we get that for all $\beta \in F$ there is $\tilde \beta_n \in U(F)$ such that 
    \[\tilde \beta_n(\omega) = \begin{cases}
        \beta(\omega) & \beta(\omega)<n\\
        0&\text{otherwise.}
    \end{cases}\]
   We have that for every $\beta\in F$:
   \[K(\G_n,F,G)\bbl\beta <n \to \beta = \tilde \beta_n\bbr=\1\]
   and by the validity of predicate logic
   \[K(\G_n,F,G)\bbl\beta <n \to (\tilde \varphi(\Gamma,\beta,\overline \alpha) \equiv \tilde \varphi(\Gamma,\tilde \beta_n,\overline \alpha))\bbr=\1,\]
   which together implies
   \[K(\G_n,F,G)\bbl\tilde \varphi(\Gamma,\tilde \beta_n,\overline \alpha) \to (\beta<n \to \tilde \varphi(\Gamma,\beta,\overline \alpha))\bbr = \1,\]
   which can be rewritten as
   \[K(\G_n,F,G)\bbl\tilde \varphi(\Gamma,\tilde \beta_n,\overline \alpha)\bbr\leq K(\G_n,F,G)\bbl \beta<n \to \tilde \varphi(\Gamma,\beta,\overline \alpha)\bbr.\tag{\dag}\]
   
   Moreover, for every $\gamma \in U(F)$ we have $\tilde \gamma_n = \gamma$ and thus every element of $U(F)$ is of the form $\tilde \beta_n$ for some $\beta$.
   
   \textbf{Claim:} For all $\overline \alpha\in U(F)$ we have:
   \begin{align*}
       \Land_{\beta\in F}K(\G_n,F,G)\bbl \beta<n \to \tilde \varphi(\Gamma, \beta,\overline \alpha)\bbr
       &=\Land_{\beta \in F} K(\G_n,F,G)\bbl \tilde \varphi(\Gamma, \tilde \beta_n,\overline \alpha)\bbr
   \end{align*}
   From $(\dag)$ we obtain
  \[
       \Land_{\beta\in F}K(\G_n,F,G)\bbl \beta<n \to \tilde \varphi(\Gamma, \beta,\overline \alpha)\bbr
       \geq\Land_{\beta \in F} K(\G_n,F,G)\bbl \tilde \varphi(\Gamma, \tilde \beta_n,\overline \alpha)\bbr.
    \] 
    For the second direction, using the fact that $U(F)$ contains exactly all $\tilde \beta_n$, we have
    \begin{align*}
       \Land_{\beta\in F}K(\G_n,F,G)\bbl \beta<n \to \tilde \varphi(\Gamma, \beta,\overline \alpha)\bbr &\leq
       \Land_{\beta\in U(F)}K(\G_n,F,G)\bbl \beta<n \to \tilde \varphi(\Gamma, \beta,\overline \alpha)\bbr \\
       &= \Land_{\beta\in F}K(\G_n,F,G)\bbl \tilde\beta_n<n \to \tilde \varphi(\Gamma, \tilde\beta_n,\overline \alpha)\bbr \\
       &= \Land_{\beta\in F}K(\G_n,F,G)\bbl \tilde \varphi(\Gamma, \tilde\beta_n,\overline \alpha)\bbr,
    \end{align*}
    this proves the claim.
    
   With the claim established, we can finish the inductive step for the universal quantifier. Again using that $U(F)$ consists exactly of all $\tilde \beta_n$ and the induction hypothesis, we have that for all $\overline\alpha \in U(F)$:
   \begin{align*}
       K(\G_n,F,G)\bbl(\forall y<n)\tilde \varphi(\Gamma,y,\overline \alpha)\bbr&=\Land_{\beta\in F}K(\G_n,F,G)\bbl \beta<n \to \tilde \varphi(\Gamma, \beta,\overline \alpha)\bbr\\
       &=\Land_{\beta \in F} K(\G_n,F,G)\bbl \tilde \varphi(\Gamma, \tilde \beta_n,\overline \alpha)\bbr\\
       &=\Land_{\gamma \in U(F)} K(\G_n,F,G)\bbl \tilde \varphi(\Gamma, \gamma,\overline \alpha)\bbr\\
       &=\Land_{\gamma \in U(F)} \lim_F \G_n\bbl\varphi(\gamma,\overline \alpha)\bbr\\
       &=\lim_{F}\G_n\bbl(\forall y)\varphi(y,\overline \alpha)\bbr. \qedhere
   \end{align*}
\end{proof}

%% file: sectfnp.tex
\section{Total NP search problems and Wide Limits}\label{sectfnp}

In this section we recall the class of search problems called Total NP Search problems, first defined in~\cite{PAPADIMITRIOU1994498}, and then take a wide limit of all instances of $\WeakPigeon$ and show that when it is intepreted as an instance, it has no solution.

\subsection{Preliminaries on Total NP Search problems}

\begin{defi}
    A total NP search problem $P(x,y)$ is a relation on binary strings for which the following two conditions are satisfied.
    \begin{itemize}
        \item \emph{Verifiability in polynomial time:} There exists a Turing machine $M$ deciding whether $P(x,y)$ in time polynomial in the lengths of $x$ and $y$.
        \item \emph{Totality:} There is a $c\in \NN$ such that for every $x\in\{0,1\}^*$ there is $y\in\{0,1\}^*$, $\abs{y}$ is at most $(\abs{x}+2)^c$ and $P(x,y)$.
    \end{itemize}

    The class of all such search problems is denoted $\TFNP$.
\end{defi}

There are many well-known $\TFNP$ problems, some of which were considered already in~\cite{PAPADIMITRIOU1994498}, such as $\LEAF$, which intuitively amounts to the task:
\begin{center}
    ``Given a graph $\omega$ with an odd-degree vertex, find another.''
\end{center} 
or $\Pigeon$ intuitively described as:
\begin{center}
 ``Given a map $f:[2^k] \to [2^k-1]$, find distinct $x$ and $y$ such that $f(x)=f(y)$.''   
\end{center}
 where $[m]$ denotes the set $\{0,\dots,m-1\}$. The graph $\omega$ in the first problem is given by a circuit computing potential neighbors for each vertex, and the function $f$ in the second problem is given by a circuit computing its values. This makes the problems non-trivial as checking the whole $\omega$ or $f$ may take exponential time in the size of the given circuit.

More relevant for our setting is the variant of the class called \emph{black-box $\TFNP$}, originally defined as `type $2$ $\TFNP$' in~\cite{beame1995rel}. This variant allows the circuits in the input to be replaced by an oracle, for example, in the problem $\LEAF$ we instead obtain oracle access to the neighbor-sets of $G$ or in the problem $\Pigeon$ to the values of the function $f$. We will start by defining a query tree, which is a natural computation model generalizing an oracle Turing machine. Query trees capture exactly the relationship between the values of the oracle and the computational states of the model. Usually such trees operate on relational oracles, see for example~\cite{goos2024separations}, but in our setting function oracles are more natural. \footnote{Moreover, function oracles, with values polynomially bounded in the size of the input, can be simulated by relational ones which can be queried for $i$-th bit of the given function value. Every query tree operating on function oracle then can be transformed to a tree operating on relational oracle with at most polynomially larger depth.}

\begin{defi}
    Let $k_q,k_a,k_o>0$. A $(k_q,k_a,k_o)$-query tree $T$ is a labeled rooted tree of the following shape:
    \begin{itemize}
        \item Each non-leaf node is labeled by a binary string of length $k_q$.
        \item Each non-leaf node has for every $w\in\{0,1\}^{k_a}$ an outgoing edge labeled by $w$.
        \item Each leaf node is labeled by a binary string of length $k_o$.
    \end{itemize}

    The depth of a tree $T$ is the length of the longest path starting at the root. If $\O_{k_q}$ is a function which maps $\{0,1\}^{k_q}$ to $\{0,1\}^{k_a}$, then the \emph{computation} of $T$ on $\O_{k_q}$ is the path obtained by starting at the root and continuing down the tree from each non-leaf node with some label $l$ to the next node along the edge labeled by $\O_{k_q}(l)$, the \emph{output} of the computation $T(\O)$ is then simply the label of the leaf met on the computation.
\end{defi}

For a function oracle $\O$ and a number $k$ we denote the restriction of $\O$ to $\{0,1\}^{k}$ as $\O_{k}$.

\begin{defi}
    A total query NP search problem $P(\O,x,y)$ is a relation, where $\O$ is a function oracle, $x$ and $y$ are binary strings, along with two functions $\ell_a$ and $\ell_s$ such that the following three conditions are satisfied.
    \begin{itemize}
        \item \emph{Conditions on lengths:} The functions $\ell_a$ and $\ell_s$ are polynomial time when the input is given in unary. For every $k$, we have $\O_k:\{0,1\}^k\to\{0,1\}^{\ell_a(k)}$ and for every $x,y\in\{0,1\}^*$ we have that $P(\O,x,y)$ implies $\abs{y}=\ell_s(\abs{x})$.
        \item \emph{Verifiability in polynomial depth:} There is a polynomial $p$, and for any binary strings $x$ and $y$ there exists a query tree $T_{x,y}$ of depth at most $p(\abs{x})$, such that for every $\O$ we have $T_{x,y}(\O)=1$ if and only if $P(\O,x,y)$. 
        \item \emph{Totality:} For each $x\in\{0,1\}^*$ there is $y\in\{0,1\}^{\ell_s(\abs{x})}$ such that $P(\O, x,y)$.
    \end{itemize}

    The pair $(\O,x)$ satisfying the conditions on lengths is called \emph{an instance of} $P$, and the string $x$ is called \emph{the size parameter}. The class of all such search problems is denoted $\TFNP^{dt}$.
\end{defi}

We will be analyzing the following two $\TFNP^{dt}$ problems. 

\begin{defi}
    The problem $\WeakPigeon$ is given as follows. Given $x$ and a function oracle $\O$ with $\O_{\abs{x}}$ $:\{0,1\}^{\abs{x}}\to\{0,1\}^{\abs{x}-1}$ find distinct $x_1,x_2\in\{0,1\}^{\abs{x}}$ such that $\O(x_1)=\O(x_2)$.
\end{defi}

This problem is total because the size of the domain on a given length is larger than the size of the codomain. Let us now assume, that the input of a $\TFNP^{dt}$ problem can be given by multiple function oracles, since these oracles can be represented as a single oracle computing all their values in parallel.

\begin{defi}
    The problem $\OntoWeakPigeon$ is given as follows. Given $x$ and function oracles $\O^f$ and $\O^g$ such that \[\O^f_{\abs{x}}:\{0,1\}^{\abs{x}} \to \{0,1\}^{\abs{x}-1} \qquad \text{ and }\qquad \O^g_{\abs{x}-1}:\{0,1\}^{\abs{x}-1}\to\{0,1\}^{\abs{x}}\] 
    find $x'\in\{0,1\}^{\abs{x}}$ satisfying that $\O^g(\O^f(x'))\neq x'$.
\end{defi}

The problem is total as a consequence of totality of $\WeakPigeon$, if we have distinct $x_1,x_2$ such that $\O^f(x_1)=\O^f(x_2)$ then one of them have to already be a solution to $\WeakPigeon$. This observation can be made precise using the notation of a many-one reduction. These reductions originally used oracle Turing machines (see~\cite{beame1995rel}), but as we already replaced oracle Turing machines by query trees we shall modify the definition accordingly.

\begin{defi}
    Let $P$ and $Q$ be $\TFNP^{dt}$ problems such that the length functions for $P$ are $\ell^P_a$ and $\ell^P_s$ and for $Q$ they are $\ell^Q_a$ and $\ell^Q_s$. We say that $P$ is many-one reducible to $Q$, denoted $P\leq_m Q$, if there is a function $r:\{0,1\}^*\to\{0,1\}^*$ computed in polynomial-time and for each $x\in\{0,1\}^*$, there is a sequence of $(\abs{x},\ell^P_a(\abs{x}), \ell^Q_a(\abs{r(x)}))$-query trees $(T^I_{x'})_{x'\in\{0,1\}^{\abs{r(x)}}}$ and a sequence of $(\abs{x},\ell^P_a(\abs{x}),\ell^P_s(\abs{x}))$-query trees $(T^s_{y'})_{y'\in\{0,1\}^{\ell^Q_s(\abs{r(x)})}}$, such that the following is satisfied:

    For every instance $(\O,x)$ of $P$, let $\O'$ denote the function oracle satisfying that $\O'_{\abs{r(x)}}(x')$ is given by the value $T^I_{x'}(\O)$. Then, for every $y'\in\{0,1\}^{\ell^Q_s(\abs{r(x)})}$ satisfying $Q(\O',r(x),y')$ we have $P(\O,x,T^s_{y'}(\O))$.
\end{defi}

It is easy to check that $\OntoWeakPigeon \leq_m \WeakPigeon$. Regarding the other direction, it follows from known results (see Section~\ref{secconc} for details) that \[\WeakPigeon \not \leq_m \OntoWeakPigeon,\] we will give a new proof of this using wide limits in the rest in the remainder of Section~\ref{sectfnp} and Section~\ref{secsepwphp}.

\subsection{The wide limit of all instances of $\WeakPigeon$}

The following wide sequence essentially consists of all instances of $\WeakPigeon$. We will show, that relative to trees of subexponential depth which are allowed to ask for neighbors of vertices, the wide limit will be a graph of an injective function.

\begin{defi}
    Let $M_{k,\floor{k/2}}$ be a wide sequence consisting of all $\{E\}$-structures on $\{0,\dots, k-1\}$, where $E$ is a graph of a function from $\{0,\dots,k-1\}$ to $\{0,\dots,\floor{k/2}-1\}$.
\end{defi}

Note, that for a fixed $\omega\in\Mnh$ the neighbor of $i\in\{0,\dots,n-1\}$ in $\omega$ is simply the image of $i$ in the function whose graph is $\omega$.

\begin{defi}
    We define $\Tnb$ as the set of all labeled rooted trees of the following shape:
    
    \begin{itemize}
    \item Each non-leaf node is labeled by some $v\in\{0,\dots,n-1\}$. 
    \item For each $u \in \{0,\dots, n-1\}$ and a node $a$ there is an outgoing edge from $a$ labeled $u$.
    \item Each leaf is labeled by some $m\in \M_n$.
    \item The depth of the tree is defined as the maximal number of edges in a path from the root, and we require it is at most $n^{1/t}$ (rounded to the nearest element of $\M$) for some $t\in\M\setminus \mathbb{N}$.
    \end{itemize}

    The \emph{computation} of such a tree in $\Tnb$ on $\omega\in M_{n,\floor{n/2}}$ is defined as follows. We build a path by starting at the root and interpreting every node labeled by some $v$ as a question `what is the neighbor of the vertex $v$?' and we follow the output edge with the answer and continue analogously until we find a leaf. The label of the leaf is defined to be the output of the computation.

    We define $\Fnb$ to be the set of all functions on $\Mnh$ which are computed by some $T\in\Tnb$. Note that the depth of the trees is subexponential in $\abs{n}$.
\end{defi}

\begin{defi}
    We say a tree $T\in \Tnb$ \emph{fails} on $\omega\in \Mnh$ if on the computation path the tree $T$ the neighbors of all distinct vertices are distinct, and if $T(\omega)\in\{0,\dots,n-1\}$ then also the neighbor of $T(\omega)$ in $\omega$ is disjoint from the other neighbors.
\end{defi}

\begin{lemm}\label{lemmabirthdayparadox}
    Let $T\in\Tnb$, then
    \[\st\left(\Pr_{\omega\in M_{n,\floor{n/2}}}[T\text{ fails}]\right) = 1.\]
\end{lemm}
\begin{proof}
    By direct computation, we have that the probability of failure of tree of depth $d$ is
    \begin{align*}
        \Pr_{\omega\in M_{n,\floor{n/2}}}[T\text{ fails}] = \prod_{i=1}^d \left (1-\frac{i}{\floor{n/2}} \right) \geq \left(1-\frac{d}{\floor{n/2}}\right)^d \geq 1-\frac{d^2}{\floor{n/2}},
    \end{align*}
    where the last inequality follows from Theorem~\ref{thrmbrnl}. Since the depth of any tree in $\Tnb$ is bounded by $n^{1/t}$, for some $t\in\M\setminus \NN$, we have that the lower bound is at least $1-((n^{2/t})/\floor{n/2})$, which is infinitesimally close to $1$.
\end{proof}

\begin{thrm}\label{thrmwlcoll}
    \[\lim_{\Fnb}\Mnh \bbl (\exists x)(\exists y)(\exists z)(x\neq y \land E(x,z)\land E(y,x))\bbr= \0\]
\end{thrm}
\begin{proof}
    For contradiction assume that 
    \[ \lim_{\Fnb}\Mnh \bbl (\exists x)(\exists y)(\exists z)(x\neq y \land E(x,z)\land E(y,x))\bbr > \0,\]
    therefore, there are $\alpha,\beta,\gamma\in U(\Fnb)$, such that there is $p>0$ and 
    \[ \mu(\lim_{\Fnb}\Mnh \bbl (\alpha\neq \beta \land E(\alpha,\gamma)\land E(\beta,\gamma))\bbr) = p,\]
    therefore, as the evaluated sentence is quantifier free, we have
    \[ \st\left(\Pr_{\omega \in \Mnh} [(\alpha(\omega)\neq \beta(\omega) \land E_\omega(\alpha(\omega),\gamma(\omega))\land E_\omega(\beta(\omega),\gamma(\omega)))]\right) = p. \tag{\dag}\]

    Assume that $T_\alpha,T_\beta,T_\gamma \in \Tnb$ are the trees computing $\alpha,\beta$ and $\gamma$. Let us consider the tree $T\in \Tnb$ which can be obtained by taking $T_\gamma$ and replacing every leaf by a copy of the tree $T_\alpha$, and then appending $T_\beta$ to every new leaf by every possible labeled edge. By $(\dag)$ the fraction of paths of $T$ where the tree asks for neighbors of $\alpha(\omega)$ and $\beta(\omega)$ and obtains $\gamma(\omega)$ is infinitesimally close to $p$ and thus the probability of failure of $T$ is not infinitesimally close to $1$. This is in contradiction with Lemma~\ref{lemmabirthdayparadox}.
\end{proof}

%% file: secsepwphp.tex
\section{Model for separation of $\OntoWeakPigeon$ and $\WeakPigeon$}\label{secsepwphp}

In this section, we will expand $\lim_{\Fnb} \Mnh$ to the model $K(\Mnh,\Fnb,\Gnb)$, and show that $\OntoWeakPigeon$ is total, but the problem $\WeakPigeon$ is not. We then show that this implies nonexistence of a many-one reduction from $\WeakPigeon$ to $\OntoWeakPigeon$.

\subsection{Construction of the model}

Let us start by assuming for the rest of this work that $n$ is a power of two, this will allow us to easily convert between sets of binary strings and numbers. In the models $\M_n$ we are working with there is a pairing function $\langle i,j \rangle$ which codes pairs of numbers by a single number. Thus, we can represent functions of any finite arity by unary functions in $\M_n$. We use this to define the family $\Gnb$.

\begin{defi}
    We define $\Gnb$ to be the family of all random functions on $\Mnh$ which fulfill the following. For each unary $\Theta \in \Gnb$ there exists a tuple $(\gamma_0,\dots,\gamma_{m-1})$ coded by an element of $\M$, such that $\gamma_i\in \Fnb$ and for $\alpha\in\Fnb$
    \[\Theta(\alpha)(\omega)=\begin{cases}
        \gamma_{\alpha(\omega)}(\omega)&\alpha(\omega)<m\\
        0&\text{otherwise.}
    \end{cases}\]

    For every $k\in\NN,k>1$ and $k$-ary $\Theta\in\Gnb$ there is a $(k-1)$-ary $\Theta'\in\Gnb$ such that for every $\alpha_0,\dots,\alpha_{k-1}\in\Fnb$ we have
    \[\Theta(\alpha)(\omega)=
        \Theta'(\alpha_0,\dots,\langle \alpha_{k-2} , \alpha_{k-1}\rangle)(\omega).\]
\end{defi}

One can also regard $k$-ary functions from $\Gnb$ as those computed by $k$-dimensional tuples (that is tuples indexed by $k$-tuples) of elements of $\Fnb$. To further explain the formalism, if we have unary $\Theta\in\Gnb$ and some $\omega\in\Mnh$ the function $\Theta_\omega$ is the function determined by the tuple $(\gamma_0(\omega),\dots,\gamma_{m-1}(\omega))$ of elements of $\M_n$  with the property that $i \mapsto \gamma_i(\omega)$ if $i<m$. Random function families defined using tuples of random variables are used frequently in~\cite{krajicek2010forcing} and, more importantly, they generalize functions computable by an oracle Turing machine. This allows us to obtain nonexistence of many-one reducibility to $\OntoWeakPigeon$ in Theorem~\ref{thrmsep}.

\begin{defi}
    We define $\Gamma$, an element of $\Gnb$, as the random function computed by the tuple $(\gamma_{\langle i,j \rangle})_{i,j=0}^{n-1}$, where $\gamma_{\langle i,j\rangle}$ is computed by a tree $T_{i,j}\in\Tnb$ of depth $1$ which queries the neighbor of $i$ and outputs $1$ if it is $j$, and otherwise it outputs $0$.
\end{defi}

\begin{lemm}\label{lemmmisc}
    \begin{enumerate}
        \item $\Fnb$ has restrictable ranges
        \item $\Fnb$ is $L_n$-closed
        \item $\Gnb$ is $\Fnb$-compatible
        \item the edge relation of the wide limit $\lim_{\Fnb}\Mnh$ is represented in $\Gnb$ by $\Gamma$.
    \end{enumerate}
\end{lemm}

\begin{proof}
    1, 2: Here we can proceed simply by relabeling the leaves of the trees computing the functions from $\Fnb$.

    3: Assume that $\Theta\in\Gnb$ is computed by a tuple $(\gamma_0,\dots, \gamma_{m-1})$ and the depth of all trees computing $\gamma_i$ is at most $n^{1/t}$ for some $t\in\M_n\setminus \NN$. For all $\alpha\in\Fnb$ we have that the tree $T$, which we construct by appending to each leaf of the tree computing $\alpha$ with label $i$ the tree $\gamma_i$, has depth at most $n^{1/t'}$ for some $t'\in\M_n\setminus \NN$ and therefore $T\in\Tnb$. The tree $T$ computes $\Theta(\alpha)$ by the definition of $\Gnb$. Hence, $\Gnb$ is $\Fnb$-compatible.

    4: By the definition of $\Gamma$ and $\Tnb$, we have for every $\alpha,\beta\in U(\Fnb)$: \[K(\Mnh,\Fnb,\Gnb)\bbl\Gamma(\alpha,\beta)\bbr=\lim_{\Fnb}\Mnh \bbl E(\alpha,\beta)\bbr. \qedhere\]
\end{proof}

The model which we will analyze in the rest of this section is $K(\Mnh,\Fnb,\Gnb)$.

\subsection{Non-totality of $\WeakPigeon$}\label{subsecwphp}

In this section, we will show that the formalization of the statement `The problem $\WeakPigeon$ is total' is not true in the model $K(\Mnh,\Fnb,\Gnb)$. Since wide limits are defined on intervals of (non-standard) numbers and $\TFNP^{dt}$ problems are defined on sets of binary strings, let us describe how these sets correspond to each other in our formalized statement. The input oracles will be represented by elements of the set sort and each set of binary strings $\{0,1\}^t$, where $t\in\M_n$, will be identified with the interval $\{0,\dots,2^{t}-1\}$. Since $n$ is a power of two, the interval $\{0,\dots,n-1\}$ reserved for the values of $\Fnb$-vertices can be identified with $\{0,1\}^{\abs{n}-1}$. We also obtain the bijection $\{0,\dots,n-1\}\cong\{t\in\M_n ; \abs{t}<\abs{n}\}$.

We will now define a formula $\varphi_\WeakPigeon(X,m,x_1,x_2)$ which is a formalization of the statement that `The values $x_1$ and $x_2$ are a solution to the $\WeakPigeon$ instance $(X,m)$.' It is more natural in the arithmetic setting to accept as instances functions of arbitrary range and then allow inputs mapped outside $R_m=\{s\in \M_n; \abs{s}<\abs{m}-1\}$ as solutions. This variant is many-one reducible to the original problem, as we can remap the inputs outside $R_m$ to a fixed value in $R_m$.

\begin{defi}
    Let $\varphi_\WeakPigeon^0(X,m,x_1,x_2)$ be the following $L_n^2$-formula:
        \[( \abs{X(x_1)} \geq \abs{m}-1) \lor (\abs{X(x_2)}\geq \abs{m}-1) \lor (x_1 \neq x_2 \land X(x_1)=X(x_2))) \]
    and $\varphi_\WeakPigeon(X,m,x_1,x_2)$ to be the following $L_n^2$-formula:
    \begin{align*}
        (m=0) \lor (\abs{x_1} < \abs{m} \land \abs{x_2} < \abs{m} \land  \varphi_\WeakPigeon^0(X,m,x_1,x_2))) 
    \end{align*}
\end{defi}

\begin{thrm}\label{thrmwphp}
   \[K(\Mnh, \Fnb, \Gnb)\bbl (\forall X)(\forall m)(\exists x_1)(\exists x_2)(\varphi_\WeakPigeon(X,m,x_1,x_2))\bbr = \0\]
\end{thrm}
\begin{proof}
    Let us fix $\Theta\in \Gnb$ which is computed by the tuple $(\theta_0,\dots,\theta_{n-1})$, where the element $\theta_i$ is computed by the depth one tree, whose root is labeled $i$ and each leaf is labeled the same as its corresponding edge.

    \textbf{Claim I:} $K(\Mnh, \Fnb, \Gnb)\bbl (\forall x<n)(\forall y<n)((\Theta(x)=y) \equiv \Gamma(x,y))\bbr = \1$\\
    \textit{Proof of claim.} Since the sentence is universal, it is enough to check its validity on every $\omega\in\Mnh$, which in turn follows from the definitions of $\Theta$ and $\Gamma$.

    \textbf{Claim II:} \[K(\Mnh, \Fnb, \Gnb)\bbl (\forall x<n)(\forall y<n)(\forall z<n)(x = y \lor \lnot \Gamma(x,z) \lor \lnot\Gamma(y,z))\bbr = \1\]
    \textit{Proof of claim.} This follows from Theorem~\ref{thrmwlcoll} which shows the validity of the corresponding sentence in the wide limit $\lim_{\Fnb} \Mnh$, Lemma~\ref{lemmmisc} and Theorem~\ref{thrmtrans}.
    
    \textbf{Claim III:} $K(\Mnh, \Fnb, \Gnb)\bbl (\forall x<n)(\forall x_2<n)(x_1 = x_2 \lor \Theta(x_1) \neq \Theta(x_2))\bbr = \1$
    \textit{Proof of claim.} By Claims I and II and validity of predicate logic.

    \textbf{Claim IV:} $K(\Mnh, \Fnb, \Gnb)\bbl (\forall x<n)(\abs{\Theta(x)}<\abs{n}-1)\bbr = \1$\\
    \textit{Proof of claim.} Again, by the universality we just need to check the validity at each $\omega\in \Mnh$, which follows as $\Theta$ outputs neigbors of vertices in $\omega$ which are by the definition of $\Mnh$ always in the range $\{0,\dots,\floor{n/2}-1\}=\{t\in M_n; \abs{t}<\abs{n}-1\}$.

    The claims III and IV show the validity of the negations of all disjuncts of the formula $\varphi_\WeakPigeon^0(\Theta,n)$. The formula $(n=0)$ is also obviously assigned the Boolean value $\0$. This implies that
    \[K(\Mnh, \Fnb, \Gnb)\bbl (\exists x_1<n)(\exists x_2<n)\varphi_\WeakPigeon(\Theta,n)\bbr = \0,\]
    which implies the theorem by the Boolean evaluation of universal quantifiers.
\end{proof}

\subsection{Totality of $\OntoWeakPigeon$}

We again start by defining a formula $\varphi_\OntoWeakPigeon(X,Y,m,x)$ which is a formalization of a statement: `The value $x$ is a solution to the instance $(X,Y,m)$ of the problem $\OntoWeakPigeon$.'

\begin{defi}
    Let $\varphi_\OntoWeakPigeon^0(X,Y,m,x)$ be the following $L_n^2$-formula:
    \begin{align*}
        (\abs{X(x)} &\geq \abs{m}-1) \lor (Y(X(x))\neq x)
    \end{align*}
    and let  $\varphi_\OntoWeakPigeon(X,Y,m,x)$ be the following $L_n^2$-formula:
    \[
        (m=0) \lor (\abs{x} <\abs{m} \land \varphi_\OntoWeakPigeon^0(X,Y,m,x)).
        \]
\end{defi}

\begin{defi}
    Let $\Theta,\Xi \in \Gnb$, $\mu\in \Fnb$ and $\omega\in\Mnh$. We say that a tree $T\in\Tnb$ \emph{fails for} $(\Theta,\Xi,\mu)$ \emph{on} $\omega$ if $\mu(\omega)\neq 0$ and either $T(\omega)\geq \abs{\mu(\omega)}$ or
    \[(\Xi_\omega(\Theta_\omega(T(\omega))) = T(\omega) \text{ and } \Theta_\omega(T(\omega)) < \abs{\mu(\omega)}-1).\]
\end{defi}

\begin{lemm}\label{lemmrwphptree}
    For every $\Theta,\Xi \in \Gnb$ and $\mu\in \Fnb$ there is a tree $T\in\Tnb$ such that
    \[\st\left(\Pr_{\omega\in\Mnh}[T\text{ fails for }(\Theta,\Xi,\mu)\text{ on $\omega$}]\right )=0.\]
\end{lemm}
\begin{proof}
    We can assume that $\Theta$ is computed by $(\theta_0,\dots,\theta_{r-1})$ and $\Xi$ by $(\xi_0,\dots,\xi_{r-1})$, such that for all $\omega\in\Mnh$: $\mu(\omega) < r$ and that the depth of all the trees computing $\theta_i,\xi_i$ and $\mu$ is bounded by $n^{1/t}$ for some $t\in \M_n\setminus \NN$. For any $\Theta$ and $\Xi$ this can be achieved by padding any pair of tuples they are computed by the constant function $0$ in $\Fnb$.

    We will recursively construct trees $T_d\in \Tnb$ indexed by elements of $\M_n$ and then show that $T=T_{d}$ for sufficiently large value of $d\in \M_n$ satisfies the statement of the lemma. We will be constructing the tree by prolonging paths of previously constructed trees, where by a path we mean a complete path of nodes from the root all the way to some leaf. We say a sample $\omega\in\Mnh$ is \emph{compatible} with a path $p$ in a tree if the computation of the tree on $\omega$ is exactly $p$.

    \textbf{The initial tree $T_0$:} The initial tree $T_0$ is the tree computing $\mu$ whose leaves we relabel as follows. For every path $p$ in $T_0$ such that all $\omega\in\Mnh$ compatible with $p$ give $\mu(\omega)=0$, we have a non-failure of $T_0$ and we can keep the original label. On the other hand, for each path $p$ in $T_0$ such that all $\omega\in\Mnh$ compatible with $p$ give $\mu(\omega)>0$ we pick a label in the following way. Let $p$ be such a path and $m$ be the value $\mu(\omega)$ obtained for any $\omega\in\Mnh$ compatible with $p$. For each $\omega\in \Mnh$ compatible with $p$ we have
    \[\Pr_{i\in\{0,\dots,m-1\}}[\abs{\Theta_\omega(i)} <  \abs{m}-1 \text{ and } \Xi_\omega(\Theta_\omega(i)) = i]\leq \frac{2^{\abs{m}-2}}{2^{\abs{m}-1}}=\frac{1}{2},\]
    because the image of $\Xi_\omega\restriction\{s\in M_n;\abs{s}<\abs{m}-1\}$ has at most $2^{\abs{m}-2}$-many elements. This implies by an averaging argument that there is $i_p\in\{0,\dots,m-1\}$ such that
    \[\Pr_{\omega\in\Mnh}[\Theta_\omega(i_p) <  \floor{m/2} \text{ and } \Xi_\omega(\Theta_\omega(i_p)) = i_p \mid \omega\text{ compatible with }p] \leq \frac{1}{2}.\]
    We relabel each leaf at such a path $p$ to $i_p$. This concludes the construction of $T_0$ and assures that $\Pr_{\omega\in\Mnh}[T_0\text{ fails}] \leq \frac{1}{2}$ and that the depth of $T_0$ is at most $n^{1/t}$.

    \textbf{The recursive step:} Let us now assume that $d>0$ and that $T_{d-1}$ was already constructed with the property that
    \[\Pr_{\omega\in\Mnh}[T_{d-1}\text{ fails}] \leq \left(\frac{1}{2}\right)^{d}\]
    and that the depth of $T_{d-1}$ is at most $2dn^{1/t}$. Moreover, the tree $T_{d-1}$ on each path $p$ determines a value $m=\mu(\omega)$ and if this path fails, it also determines two $(d-1)$-element sets $D\subseteq \{s\in M_n; \abs{s}<\abs{m}\}$ and $R\subseteq \{s\in M_n; \abs{s}<\abs{m}-1\}$ such that $\Theta_\omega\restriction D$ and $\Xi_\omega \restriction R$ form a pair of inverse functions on every $\omega$ compatible with $p$. We also assume, that the leaf of every path $p$ is labeled by some element of $(\{s\in M_n; \abs{s}<\abs{m}\}\setminus D)$.  We will prolong $T_{d-1}$ along each path $p$ in two stages and relabel the leaves to obtain the tree $T_d$.

    \emph{First prolongation:} 
    Let $p$ be a path in $T_{d-1}$.
    If the tree does not fail on any sample $\omega\in\Mnh$ compatible with $p$, we keep $p$ as it is. Otherwise, we replace the leaf by the tree computing $\theta_{l_p}$ and change the label of each new leaf from $l$ to the tuple $(l_p,l)$ and call all such new paths \emph{active}. We call the resulting tree $T_{d-1}'$ and by the assumption on the tree computing $\xi_{l_p}$ we get its depth is at most $n^{1/t}+2dn^{1/t}$.

    \emph{Second prolongation:}
    Let $p$ be an active path in $T_{d-1}'$ with its leaf labeled $(l_p,l_p')$ and let $m=\mu(\omega)$ on any $\omega$ compatible with $p$. If $\abs{l_p'} \geq \abs{m}-1$ then after relabeling the leaf of $p$ to the value $l_p$, the tree does not fail on any $\omega\in\Mnh$ compatible with $p$. If $\abs{l_p'} < \abs{m}-1$, we replace the leaf by the tree $\xi_{l_p'}$ and change the label of each new leaf from $l$ to the triple $(l_p,l_p',l)$. We will denote the resulting tree $T_{d-1}''$ and again call its newly prolonged paths \emph{active}. By the assumptions on the tree computing $\xi_{l_p'}$ we have that the depth of $T_{d-1}''$ is at most $2n^{1/t}+2dn^{1/t}=2(d+1)n^{1/t}$.

    \emph{Relabeling:} We will now relabel the leaves of active paths of $T_{d-1}''$. Let $p$ be an active path of $T_{d-1}''$ and let its leaf be labeled $(l_p,l_p',l_p'')$. If $l_p''$ is distinct from $l_p$ we can relabel the leaf to the value $l_p$, this makes the tree never fail along this path.  Otherwise, $l_p\not \in D$ and the tree so far have established a bijection between $D\cup\{l_p\}$ and $R\cup\{l_p'\}$ computed by $\Theta_\omega$ and $\Xi_\omega$ on any $\omega\in\Mnh$ compatible with $p$. This implies, that for each $\omega\in\Mnh$ compatible with $p$ we have
    \[\Pr_i[\abs{\Theta_\omega(i)} < \abs{m}-1 \text{ and } \Xi_\omega(\Theta_\omega(i)) = i]\leq \frac{2^{\abs{m}-2}-d}{2^{\abs{m}-1}-d} \leq \frac{2^{\abs{m}-2}}{2^{\abs{m}-1}} = \frac{1}{2},\]
    where the $i$ is sampled uniformly in $\{s\in M_n; \abs{s}<\abs{m}\}\setminus (D \cup \{l_p\})$ which by an averaging argument implies the existence of $i_p$ such that
    \[\Pr_{\omega\in\Mnh}[\abs{\Theta_\omega(i_p)} <  \abs{m}-1 \text{ and } \Xi_\omega(\Theta_\omega(i_p)) = i_p \mid \omega\text{ compatible with }p] \leq \frac{1}{2},\]
    we then relabel the leaf of $p$ to $i_p$. After relabeling all active paths in this way, we obtain the tree $T_d$.

    \emph{Properties of $T_d$:} The depth of $T_d$ is the same as of $T_{d-1}''$ which is at most $2(d+1)n^{1/t}$. During the relabeling, we have also established that the failing paths determine a bijection with one additional element added to both $R$ and $D$ of the corresponding path in $T_{d-1}$ and that the leaves are properly labeled. It remains to analyze the probability of $T_d$ failing. If a path $p$ in $T_{d-1}$ did not fail on some $\omega\in\Mnh$, any of its extensions in $T_d$ also did not fail on $\omega$. Moreover, if $p$ did fail, we found a new label for each of its extensions in $T_d$, which fails with probability at most $1/2$. Therefore,
    \[\Pr_{\omega\in\Mnh}[T_d\text{ fails}] \leq \frac{1}{2}\cdot\Pr_{\omega\in\Mnh}[T_{d-1}\text{ fails}] \leq \left (\frac{1}{2}\right ) ^d.\]

    \textbf{Conclusion:} Let $d=n^{1/t}$, then the depth of $T_d$ is at most $n^{1/(t/2)}$. Moreover,
    \[\Pr_{\omega\in\Mnh}[T_d\text{ fails}]\leq \left (\frac{1}{2}\right)^{n^{1/t}+1},\]
    standard part of which is $0$. Thus, $T=T_d$ satisfies the statement of the lemma.
\end{proof}

\begin{thrm}\label{thrmrwphp}
   \[K(\Mnh, \Fnb, \Gnb)\bbl (\forall X)(\forall Y)(\forall m)(\exists x<m)(\varphi_\OntoWeakPigeon(X,Y,m))\bbr = \1\]
\end{thrm}
\begin{proof}
   Let $\Theta,\Xi\in\Gnb$ and $\mu\in\Fnb$, we will find $\gamma\in\Fnb$ such that 
   \[K(\Mnh, \Fnb, \Gnb)\bbl \varphi_\OntoWeakPigeon(\Theta,\Xi,\mu,\gamma)\bbr = \1,\]
   which is enough to prove the theorem.

   Let $\gamma\in\Tnb$ be computed by the tree whose existence follows from Lemma~\ref{lemmrwphptree} for $\Theta,\Xi$ and $\mu$. The formula $\varphi_\OntoWeakPigeon(\Theta,\Xi,\mu,\gamma)$ is interpreted for each $\omega\in\Mnh$ as the statement claiming the tree computing $\gamma$ not failing for $(\Theta,\Xi,\mu)$ on $\omega$. Since it is a quantifier-free formula, it can be evaluated in $K(\Mnh,\Fnb,\Gnb)$ by calculating its probability over the sample space which by Lemma~\ref{lemmrwphptree} is infinitesimally close to $1$. Therefore, the formula $\varphi_\OntoWeakPigeon(\Theta,\Xi,\mu,\gamma)$ is valid in $K(M_{n,\floor{n/2}},\Fnb,\Gnb)$.
\end{proof}

\subsection{Nonreducibility of $\WeakPigeon$ to $\OntoWeakPigeon$}

In the next theorem, we will be working both with non-standard binary strings and non-standard numbers. We will define two pairs of functions which convert between these two types of objects. The first pair consists of the $\textsc{Wrap}$ function $\w$ which converts binary strings to non-zero numbers by appending a leading $1$ and interpreting the string as a binary expansion of the resulting number. The inverse to $\w$ is the $\textsc{Unwrap}$ function denoted $\u$. These are then used to convert between the size parameter of the formula formalizing totality, either $\varphi_\WeakPigeon$ or $\varphi_\OntoWeakPigeon$, and the size parameter of the associated $\TFNP^{dt}$ problem.

The second pair then provides an explicit bijection between the set $\{s;\abs{s}<\abs{m}\}$ and $\{0,1\}^{\abs{\u(m)}}$, or alternatively between $\{s; \abs{s}< \abs{\w(x)}\}$ and $\{0,1\}^{\abs{x}}$, depending on whether we start with a binary string $x$ or an element $m\in \M_n$. The first function, $\textsc{Pad}$, denoted $\p$ simply takes an input number $s$ and outputs its binary expansion with enough leading zeros to obtain a string of required size, this required size is a second argument of the function. However, we will never write this second argument explicitly, and its value will always be clear from the context. The function $\textsc{Unpad}$, denoted $\up$, is the inverse to $\p$ and simply takes a binary string and interprets it as a binary expansion of a number ignoring leading zeros.
\begin{center}
\begin{tabular}{|c|c|c|c|c|c|}
    \hline
    Operation & Symbol & Domain & Range & Input & Output \\ \hline
   \textsc{Wrap}& $\w$ & binary strings &$\M_n\setminus \{0\}$ & $x$ & $2^{\abs{x}} +\sum_{i<\abs{x}} x_i \cdot 2^{i}$ \\ \hline
   \textsc{Unwrap}& $\u$ & $\M_n\setminus \{0\}$ & binary strings & $s$ & $\w^{-1}(s)$ \\ \hline
   \textsc{Pad}& $\p$ & $\M_n$  & binary strings & $s, l$ & \multicolumn{1}{p{3.6cm}|}{The binary expansion of $s$ padded with leading zeros to the length of $l$.}  \\ \hline
   \textsc{Unpad}& $\up$ & binary strings & $\M_n$ & $x$ & $\sum_{i<l} x_i\cdot  2^i$ \\ \hline
\end{tabular}
\end{center}

We will begin by proving the following lemma, its proof is straightforward and it essentially shows that the formulas $\varphi_\WeakPigeon$ and $\varphi_\OntoWeakPigeon$ properly formalize the fact that a given element is a solution to the respective search problem. We only state one direction for each of the problems, this simplifies the proof and exactly captures the properties needed for the proof of our separation.

\begin{lemm}\label{lemmarbin}
    \leavevmode
    \begin{enumerate}
        \item Let $\Delta\in\Gnb$, $\nu\in \Fnb$ and $\omega\in\Mnh$. If $\nu(\omega)>0$, there is a $\WeakPigeon$ instance $I=(B(\Delta_\omega),\u(\nu(\omega)))$, such that:

            For every solution $(x_1,x_2)$ to $I$, we have that 
            \[\M_n \models \varphi_\WeakPigeon(\Delta_\omega,\nu(\omega),\up(x_1),\up(x_2)).\]
        \item Let $\Theta,\Xi\in\Gnb$, $\mu\in \Fnb$ and $\omega\in\Mnh$. If 
        \begin{enumerate}
            \item $\mu(\omega)>0$,
            \item the range of $\Theta_\omega\restriction \{s; \abs{s}<\abs{\mu(\omega)}\}$ is a subset of $\{s;\abs{s}<\abs{\mu(\omega)}-1\}$,
            \item the range $\Xi_\omega \restriction \{s; \abs{s}<\abs{\mu(\omega)}-1\}$ is a subset of $\{s;\abs{s}<\abs{\mu(\omega)}\}$,
        \end{enumerate}
        then the instance $J=(\p(\Theta_\omega(\up(-))),\p(\Xi_\omega(\up(-))),\u(\mu(\omega)))$ of the problem\\ $\OntoWeakPigeon$ satisfies: 
             
        For every $s\in \M_n$ satisfying $\M_n \models \varphi_\OntoWeakPigeon(\Theta_\omega,\Xi_\omega,\mu(\omega),s)$ we have that $\p(s)$ is a solution to $J$.
    \end{enumerate}
\end{lemm}
\begin{proof}
    (1.) The function oracle $B(\Delta_\omega)$ can be defined as follows.
    \[B(\Delta_\omega)(x')=\begin{cases}
        \p(\Delta_\omega)(\up(x'))&\abs{\Delta_\omega(\up(x'))}<\abs{\mu(\omega)}-1,\\
        \p(0) & \text{otherwise,}
    \end{cases}\]
    any solution $(x_1,x_2)$ to $I$ is then either translated to a pair of distinct numbers mapped to the same element by $\Delta_\omega$ or one of them is mapped outside $\{s;\abs{s}<\abs{\mu(\omega)}-1\}$ and hence $\M_n \models \varphi_\OntoWeakPigeon(\Theta_\omega,\Xi_\omega,\mu(\omega),s)$.
    
    (2.) This follows straight from the assumptions and the definition of the formula  $\varphi_\OntoWeakPigeon$.
\end{proof}

\begin{thrm}\label{thrmsep}
    There is no many-one reduction from the problem $\WeakPigeon$ to the problem $\OntoWeakPigeon$.
\end{thrm}
\begin{proof}
    We will proceed by contradiction, assume that $\WeakPigeon$ is many-one reducible to $\OntoWeakPigeon$. We will show that this implies 
    \[K(\Mnh,\Fnb,\Gnb)\bbl(\forall X)(\forall m)(\exists x_1)(\exists x_2)\varphi_\WeakPigeon(X,m,x_1,x_2)\bbr = \1,\]
    which is in contradiction with Theorem~\ref{thrmwphp}.
 
    The existence of the reduction implies that there is a function $r$ from binary strings to binary strings and for each length-parameter to $\WeakPigeon$ we get the suitable sequences of query trees. Let us also define a variant of $r$ which operates on numbers: The function $r'$ maps any non-zero $x\in\M_n$ to $\w(r(\u(x)))$.

    Let $\Delta \in \Gnb$, $\nu \in \Fnb$. We will use the reduction to construct some $\Theta,\Xi, \Lambda^1,\Lambda^2 \in \Gnb$ and $\mu\in \Fnb$ such that the formula    
    \[\varphi_\OntoWeakPigeon(\Theta,\Xi,\mu,\gamma) \to \varphi_\WeakPigeon(\Delta,\nu,\Lambda^1(\gamma),\Lambda^2(\gamma))\tag{\dag}\]
    is valid in $K(\Mnh,\Fnb,\Gnb)$ for any $\gamma\in\Fnb$.

    Let us assume $\Delta$ is computed by $(\delta_0,\dots,\delta_{d-1})$, $d\in \M_n$, where each element is from $\Fnb$ and $\delta_i$ is computed by $T^\delta_i\in\Tnb$ for each $i<d$. First, we define $\mu$ as the function computed by the tree of $\nu$, where we change the label of each leaf from $l$ to $r'(l)$. By Lemma~\ref{lemmarbin}, after the computation of the tree computing $\nu$ is determined, a set of $\WeakPigeon$ instances of the form $(B(\Delta_\omega),u(\nu(\omega)))$ is determined, where $\omega\in\Mnh$ is compatible with the computation of $\nu$. By the existence of the reduction, we also obtain the sequences of query trees $(T^I_{x'})_{x'\in\{0,1\}^{\abs{r(\u(m))}}}$ and $(T^s_{y'})_{y'\in\{0,1\}^{\abs{r(\u(m))}}}$.

    Our goal now is to define sequences of $\Tnb$ trees which will define $\Theta$, $\Xi$, $\Lambda_1$ and $\Lambda_2$. To do this, we first have to convert the query trees to $\Tnb$ trees. Assume, that $\mu$ is determined, therefore $\nu$ is determined and we obtain the sequences $(T^I_{x'})_{x'}$ and $(T^s_{y'})_{y'}$. For each tree $T^I_{x'}$ we define a $\Tnb$ tree $S^I_{x'}$ by simulating each query $q$ with $T^\delta_{\up(q)}$. That is, layer by layer, starting at the parents of leaves, replace each non-leaf node labeled $q$ with the tree computing $T^\delta_{\up(q)}$ and appending the subtrees from previous layer onto the leaves of the new copy $T^\delta_{\up(q)}$ with the same label as the edge they were originally incidental to. This can also be done with $T^s_{y'}$ to obtain the $\Tnb$ tree $S^s_{y'}$. Let $d'\in \M_n$ be a power of $2$ larger than any possible value of $\mu$. For each $i<d'$ we will construct the tree $U^I_i$ as follows: Start with the tree computing $\mu$ where each path $p$ in this tree with a leaf labeled $m$ satisfying $\abs{i}<\abs{m}$ gets replaced by the tree $S^I_{\p(i)}$. Analogously, we also define $U^s_i$ to be the tree constructed from the tree computing $\mu$ by appending $S^s_{\p(i)}$ onto each path labeled by $m$ satisfying $\abs{i}<\abs{m}$. Let us also assume, that the labels of $U^I_{x'}$ and $U^s_{y'}$ are numbers obtained from the original labels by the function $\up$.

    Since $\OntoWeakPigeon$ is a problem presented by two function oracles, there is a way to compute either of the functions $\O^f$ and $\O^g$ using a single input oracle which in our case is itself computed by the sequence $(T^I_{x'})_{x'}$. We will define $\Theta$, $\Xi$, $\Lambda_1$ and $\Lambda_2$ by specifying the elements of the tuples they are computed by which are denoted by the respective lowercase letters. We define $\theta_{i}$ to be computed by the tree $U^I_{i}$, except we relabel its leaves to only keep the value of $\O^f$ and similarly let $\xi_{i}$ be also computed by $U^I_{i}$ except we relabel it to only keep the value of $\O^g$. We also define $\lambda_i^1$ to be computed by $U^s_{i}$ relabeled to keep only the value for the first part of the solution and $\lambda_i^2$ to be computed by $U^s_{i}$ relabeled to keep only the second part of the solution.

    To prove the validity of $(\dag)$, which is a quantifier-free sentence, we can simply check the validity over every $\omega\in\Mnh$. After $\omega$ is fixed, the tuples 
    \[J=(\p(\Theta_\omega(\up(-))),\p(\Xi_\omega(\up(-))),u(\mu(\omega)))\text{ and }I=(B(\Delta_\omega),u(\nu(\omega)))\]
     are instances of the problems $\OntoWeakPigeon$ and $\WeakPigeon$ respectively. Assume that the antecedent of $(\dag)$ is satisfied, therefore \[\M_n\models \varphi_\OntoWeakPigeon(\Theta_\omega,\Xi_\omega,\mu(\omega),\gamma(\omega)).\]
    
    Notice that $J$ satisfies the assumptions (a), (b) and (c) of Lemma~\ref{lemmarbin} by contruction, this shows that that $\p(\gamma(\omega))$ is a solution to $J$. By the existence of the reduction, we get that $\p(\lambda^1_{\gamma(\omega)}(\omega))$ and $\p(\lambda^2_{\gamma(\omega)}(\omega))$ form a solution to $I$. And finally, the Lemma~\ref{lemmarbin} implies that 
    \[\M_n \models \varphi_\WeakPigeon(\Delta(\omega),\nu(\omega),\Lambda^1_\omega(\gamma(\omega)),\Lambda^2_\omega(\gamma(\omega))),\]
    which concludes the proof of $(\dag)$.

    From the validity of every instance of $(\dag)$, Theorem~\ref{thrmrwphp} and validity of predicate logic we obtain
    \[K(\Mnh,\Fnb,\Gnb)\bbl(\forall X)(\forall m)(\exists x_1)(\exists x_2)\varphi_\WeakPigeon(X,m,x_1,x_2)\bbr = \1.\qedhere\]
\end{proof}

%% file: sectheory.tex
\subsection{Further properties of the constructed model}

In the rest of this section we state which additional principles hold in the model.

\begin{thrm}\label{thrmopenind}
    Let $\varphi(x)$ be an open $L_n^2$-formula with parameters from $\Fnb$ and $\Gnb$. Then for every $m\in\M_n$ the open comprehension principle
    \[(\exists X)(\forall y<m)(X(y)\equiv \varphi(y))\]
    and the open induction principle
    \[\lnot \varphi(0) \lor \varphi(m) \lor (\exists x<m)(\varphi(x)\land \lnot \varphi(x+1))\]
    are both valid in $K(\Mnh,\Fnb,\Gnb)$, where we interpret $m$ as the constant function always outputting $m$ in $\Fnb$.
\end{thrm}
\begin{proof}
    This can be proven completely analogously to~\cite[Lemma 20.2.5]{krajicek2010forcing}, we will therefore only describe the basic idea. Let $m\in\M_n$ and let $\varphi(x)$ be an open $L_n^2$ formula with a single free variable with parameters from $\Fnb$ and $\Gnb$. For each number $i\in \M_n, i<m,$ the validity of $\varphi(i)$ for a given $\omega$ can be $i<t$ decided by a $\Tnb$ tree $T_i$ and the depth of all such trees is bounded above by $n^{1/t}, t\in \M_n\setminus \NN$. A tuple of elements of $\Fnb$ each computed by the corresponding tree then determines an element of $\Gnb$ which satisfies the comprehension principle for $\varphi(x)$ and $m$. The induction principle follows directly from the comprehension principle by an application of a binary search procedure.
\end{proof}

Let us give a remark about the logical strength of the separation we gave. The~language $L_n$ contains names for all the functions on $\NN$ which are polynomially bounded (the length of the output is bounded by a polynomial in the length of the input). This allows us to interpret the language $L_{\PV(f)}$, containing the names for every polynomial time algorithm with an oracle access to some polynomially bounded function $f$, in the structure $\M_n$ and therefore in $K(\Mnh,\Fnb,\Gnb)$.  We define $\forall \PV_1(f)$ to be the theory consisting of all universal $L_{\PV(f)}$-formulas true for any interpretation of $f$. The function $f$ may represent a function oracle, hence symbols from $L_{\PV(f)}$ represent (the functions computed by) the polynomial time oracle machines with access to $f$ and $\forall \PV_1(f)$ is simply the set of all universal sentences true for such functions for any oracle.

\begin{thrm}\label{thrmpv}
    Let the function symbol $f$ be interpreted in $K(\Mnh,\Fnb,\Gnb)$ by any $\Theta\in\Gnb$. Then, all the axioms of $\forall\PV_1(f)$ are valid in $K(\Mnh,\Fnb,\Gnb)$.
\end{thrm}
\begin{proof}
    The theory $\forall\PV_1(f)$ consists of all true universal statements for any interpretation of $f$. Therefore, for any $\omega\in\Mnh$, any axiom of $\forall \PV_1(f)$ is true in $\M_n$, after $f$ is interpreted as $\Theta_\omega$, under any substitution of the universal quantifiers. Thus, any axiom of $\forall \PV_1(f)$ is also valid in $K(\Mnh,\Fnb,\Gnb)$.
\end{proof}

This theorem is relevant for interpretation of our construction from the point of view of bounded arithmetic. The theory $\forall\PV_1(f)$ is the true universal extension of the relativized $\PV_1$. The theory $\PV_1$ was first defined in~\cite{kpt1991} as a first order extension of Cook's equational theory $\PV$~\cite{cook1975}.

%% file: secconc.tex
\section{Concluding remarks}\label{secconc}

In this paper we have developed some basic theory around the new concept of a limit object, the \emph{wide limit}. In the author's view, it provides a more elementary counterpart to Krajíček's structures $K(F,G)$ into which they can be expanded while preserving the values of sentences restricted to their universe (Theorem~\ref{thrmtrans}). In particular, even though the ideas which are needed to prove our many-one separation (Theorem~\ref{thrmsep}) could be collected to a direct proof of this separation, the wide limit provides a concrete semantic interpretation of this separation: An observer, who queries the images of a random map $\omega\in\Mnh$ only polynomially many times, will observe with high probability a map without a collision. This map is embodied by our limit object $\lim_{\Fnb} \Mnh$.

The problem $\OntoWeakPigeon$ and its variations have been investigated in the context of $\TFNP$ ($rPHP^{2n}_{n}$, $rWPHP$ in ~\cite{Jerabek2009},\cite{muller2021}, `lossy code' in~\cite{korten2022}) and its relation to $\WeakPigeon$ follows from well-known results. A nonreducibility of $\WeakPigeon$ to a problem stronger than $\OntoWeakPigeon$, namely $\textsc{RetractionPigeon}$, follows from a lower bound on the Sherali-Adams proof system in~\cite{dantchev2009tight} and established connections between propositional proofs and total NP search problems (see~\cite{goos2024separations}).

The main reason why was the problem $\OntoWeakPigeon$ total in the constructed model was that it had many solutions which could be checked in a constant number of queries. It would be interesting to see how this can be generalized for other total search problems.

\begin{prob}
    Characterize which $\TFNP$ problems are total in $K(\Mnh,\Fnb,\Gnb)$. 
\end{prob}

In~\cite{muller2021} the totality of certain total search problems is relatively unprovable even while preserving quite a lot of induction. In our model $K(\Mnh,\Fnb,\Gnb)$ only $\forall\PV_1(f)$ was verified. It would be interesting to investigate further how much induction can be verified in the models emerging as expansions of wide limits. In~\cite{krajicek2010forcing} there are several models of the theory $V^0_1$ constructed. However, these are obtained by expanding the family $F$ and using a suitable switching lemma to obtain a form of quantifier elimination, which in our case makes the wide limit interpretation unnatural because for this construction a more complex sample space is used. 

A general direction of further research, which would be particularly interesting, would be to characterize valid sentences of some concrete wide limit without the direct construction and thus to prove upper and lower bounds for the family $F$ of functions being considered.